\documentclass[10pt]{amsart}

\usepackage[english]{babel}
\usepackage[latin1]{inputenc}
\usepackage{amsmath}
\usepackage{amsthm}
\usepackage{amsfonts}
\usepackage{amssymb, latexsym, graphics, showidx}
\usepackage{color}

\numberwithin{equation}{section}

\newtheorem{thm}{Theorem}[section]
\newtheorem{teo}{Theorem}[section]
\newtheorem{rem}[thm]{Remark}
\newtheorem{cor}[thm]{Corollary}

\newtheorem{pro}[thm]{Proposition}
\newtheorem{lem}[thm]{Lemma}

\newtheorem{hyp}{Assumption}

\renewcommand{\dim}{\begin{proof}}

\newcommand{\finedim}{\end{proof}}
\newcommand{\R}{\mathbb{R}}
\newcommand{\Tt}{\mathbb{T}}
\newcommand{\N}{\mathbb{N}}
\newcommand{\Z}{\mathbb{Z}}

\newcommand{\al}{\alpha}

\newcommand{\ep}{\epsilon}

\newcommand{\di}{\text{div}}
\newcommand{\intor}{\int_{\Tt^N}}

\newcommand{\Rr}{{\mathbb{R}}}

\newcommand{\beq}{\begin{equation}}
\newcommand{\eeq}{\end{equation}}
\newcommand{\beqs}{\begin{equation*}}
\newcommand{\eeqs}{\end{equation*}}
\newcommand{\beqa}{\begin{eqnarray}}
\newcommand{\eeqa}{\end{eqnarray}}
\newcommand{\beqas}{\begin{eqnarray*}}
\newcommand{\eeqas}{\end{eqnarray*}}

\begin{document}

\title{Weakly coupled mean-field game systems}

\author{Diogo A. Gomes}
\address[D. A. Gomes]{
        King Abdullah University of Science and Technology (KAUST), CEMSE Division , Thuwal 23955-6900. Saudi Arabia.}
\email{diogo.gomes@kaust.edu.sa}
\author{Stefania Patrizi}
\address[S. Patrizi]{University of Texas at Austin, Austin, Texas, USA}
\email{spatrizi@math.utexas.edu}


\thanks{
        D. Gomes was partially supported by KAUST baseline and start-up funds.
        S. Patrizi was partially supported by NSF grant DMS-1262411 ``Regularity and stability results in variational problems". 
}
\date{\today}
\maketitle

\begin{abstract}
Here, we prove the existence  of solutions to first-order mean-field games (MFGs) arising in optimal switching. 
First, we use the penalization method to construct approximate solutions. Then, we prove uniform estimates for the penalized problem. 
Finally, by a limiting procedure, we obtain  solutions to the MFG problem.  
\end{abstract}

\section{Introduction}

The mean-field game (MFG)\  framework \cite{Caines2,Caines1,  ll1, ll2, ll3, ll4}
is a class of methods used to study  
large populations of rational, non-cooperative agents.
MFGs have been the focus of intense research, 
 see, for example, the surveys \cite{GPV, GS}.
Here, we investigate MFGs that arise in optimal switching. These games are given by a weakly coupled system of Hamilton-Jacobi equations of the obstacle type and a corresponding system of transport equations. 

To simplify the presentation, we use periodic boundary conditions. Thus, the spatial domain is the $N$-dimensional
flat torus, $\Tt^N$. 
Our MFG is determined by a value function, $u:\Tt^N\to\Rr^d$, a probability density, $\theta:\Tt^N\to(\Rr^+)^d$, and a
 switching current, $\nu,$ that together satisfy the following system of variational inequalities:
\begin{equation}\label{obstacleepmfg1}
\max\left(H^i(Du^i,x)+u^i- g(\theta^i),\max_j\left(u^i-u^j-\psi^{ij}\right)
\right)=0
\end{equation}
%
coupled with the system\begin{equation}\label{obstacleepmfg2}
-\di (D_p H^i(Du^i,x)\theta^i)+\theta^i+\sum_{j\neq i}\left(\nu^{ij}-\nu^{ji}\right)=1.
\end{equation}
Moreover, for 
$1\leq i,j\leq d$, 
$\nu^{ij}$ is a non-negative measure on $\Tt^N$\ supported in the set $u^i-u^j-\psi^{ij}=0$.

This system models a stationary population of agents. Each agent moves in $\Tt^N$ and can switch between different modes that are given by the index $i$. Their actions seek to minimize a certain cost. Agents can change their state by continuously modifying their
spatial position, $x$, and by switching between different modes, $i$ to $j$, at a cost
$\psi^{ji}$. The function $u^i(x)$ is the value function for an agent whose spatial location is  $x$ and whose mode is $i$. The function $\theta^i(x)$ is the density of the agents on $\Tt^N\times \{1, \hdots, d\}$. Thus,  we require that  $\theta^i(x)\geq 0$. We note that $\theta^i$ is not a probability measure on $\Tt^N\times \{1, \hdots, d\}$ because the source term in the right-hand side of \eqref{obstacleepmfg2} is not normalized.

In Section \ref{opmfg}, we discuss
detailed assumptions on the Hamiltonians $H^i$, on the nonlinearity $g$, and on the switching costs $\psi^{ij}$. A concrete example that satisfies those is
\begin{equation}
\label{HHYP}
H^i(x,p)=\frac{|p|^2}{2}+V^i(x),  \qquad g(\theta)=\ln \theta, \quad \text{and}\quad \psi^{ij}(x)=\eta, 
\end{equation}
with $V^i:\Tt^N\to\Rr$ being a $C^\infty$ function and $\eta$ being a positive real number.
Another case of interest is the polynomial nonlinearity,
 $g(m)=m^\al$ for $\alpha>0$. 

Standard MFGs involve two equations, a Hamilton-Jacobi equation and a transport
or Fokker-Planck equation. This latter equation is the adjoint of the linearization
of the former.
Because the non-linear operator in \eqref{obstacleepmfg1} is non-differentiable,
\eqref{obstacleepmfg2} is obtained by a limiting procedure. 
 In the context of MFGs, 
this method was first used in \cite{GPat}.
%
Here, we consider the following penalized problem. 
%
 \begin{equation}\label{obstacleepmfg1sys}
H^i(Du^i,x)+u^i+\sum_{j\ne i}\beta_\ep(u^i-u^j-\psi^{ij}) =g(\theta^i)\
\end{equation}
\begin{equation}\label{obstacleepmfg2sys}
-\di (D_p H^i(Du^i,x)\theta^i)+\theta^i+\sum_{j\ne i}\beta'_\ep(u^i-u^j-\psi^{ij})
\theta^i -\beta'_\ep(u^j-u^i-\psi^{ji}) \theta^j =1, 
\end{equation}
where the penalty function, $\beta_\epsilon$, is an increasing $C^\infty$ function
and $\epsilon>0.$ We assume that, as $\epsilon \to 0$, $\beta_\epsilon(s)\to \infty$ for $s>0$ and $\beta_\epsilon(s)=0$ for $s\leq 0$, see Assumption \ref{B1}. The study of optimal switching has a long history that predates
 viscosity solutions and, certainly, MFGs,  see, for example
\cite{ Belb-81,CaEv-84, CaMM-83,  EvFr-79}. In those references, the use of a penalty to
approximate a non-smooth Hamilton-Jacobi equation is a fundamental tool. The penalty in \eqref{obstacleepmfg1sys} is similar to the ones in
the aforementioned references. 

More recently, several authors have investigated weakly coupled Hamilton-Jacobi equations \cite{HT}, the corresponding extension of the weak KAM and Aubry-Mather theories \cite{CGT2, DaZa-14, MSTY-15}, the asymptotic behavior of solutions \cite{CGMT, CLLN-12, MiTr-14b, MiTr-12, Nguy-14}, and homogenization  \cite{MiTr-14a}.
In these references, the state of the system has different modes, and a random process   drives the switching between them. In contrast, here, the switches occur at deterministic times. 
Thus, our models are the MFG counterpart of the Hamilton-Jacobi systems considered in 
\cite{FFG,GoSe-11}.
MFGs with different populations \cite{MR3333058, cirant2}
are a limit case of  \eqref{obstacleepmfg1}-\eqref{obstacleepmfg2}. This can be seen by taking the limit  $\psi^{ij}\to +\infty$; that is, the case where agents are not allowed to change their state.

The development of the existence and regularity theory for MFGs has seen substantial progress in recent years. Uniformly elliptic and parabolic MFGs
are now well understood, and   the existence of smooth and weak solutions
has been  established  in a   broad range of problems, see, respectively,  \cite{GPatVrt,
GPim2, GPim1,GM, GPM1, GPM2,GPM3} and \cite{cgbt, porretta, porretta2}.
However, the regularity theory for first-order MFGs is less developed
and, in general,  only weak solutions are known to exist \cite{Cd1, Cd2, MR3358627}. Variational inequality methods are at the heart of a new class of techniques to establish the existence of weak solutions, both for first- and second-order problems  \cite{ FG2} and for their numerical approximation \cite{AFG}.
Some  MFGs that arise in applications, such as congestion 
\cite{GMit, GVrt2}
or obstacle-type problems \cite{GPat}, feature singularities. Thus, there is keen interest in developing methods for their analysis. To the best of our knowledge, this paper is the first to address MFGs arising in optimal switching. Moreover, our techniques contribute to better understanding of the regularity of first-order MFGs. 

Our main result is the following theorem. 
\begin{teo}
\label{main1}
Suppose that Assumptions \ref{A1}-\ref{A4} (see Section \ref{opmfg})
hold and that either
\begin{itemize}
\item[-]  Assumption \ref{A5} {\bf L} or
\item[-]  Assumptions \ref{A5} {\bf P-$\frac 2 N$}, \ref{A6} and \ref{A7}

\end{itemize}
hold. Then, there exists a solution, $(u, \theta) $, of \eqref{obstacleepmfg1}-\eqref{obstacleepmfg2}, with $u\in (W^{2,2}(\Tt^N))^d\cap (C^\gamma(\Tt^N))^d$ for some $\gamma\in (0,1)$ and $\theta\in (W^{1,2}(\Tt^N))^d$.
\end{teo}

 As mentioned before, to
prove the existence of solutions for \eqref{obstacleepmfg1}-\eqref{obstacleepmfg2},  we first examine the existence of solutions for  \eqref{obstacleepmfg1sys}-\eqref{obstacleepmfg2sys}, prove $\epsilon$\ independent bounds and, subsequently, consider the limit $\epsilon \to 0 $. On the existence of solutions,  our main result is the following theorem. 
\begin{teo}
\label{main2}
Suppose that Assumptions \ref{A1}-\ref{A4} (see Section \ref{opmfg}) and \ref{B1} hold, and either

\begin{itemize}
\item[-]  Assumption \ref{A5} {\bf L} or
\item[-]  Assumptions \ref{A5} {\bf P-$\frac 2 N$}, \ref{A6} and \ref{A7} 
\end{itemize}
hold. Then,  there exists a unique solution, $(u, \theta)\in (C^\infty(\Tt^N))^d\times  (C^\infty(\Tt^N))^d$, of \eqref{obstacleepmfg1sys}-\eqref{obstacleepmfg2sys} with $\theta ^i\geq \theta_0>0$ for some constant $\theta_0$ that does not depend on $\epsilon$.  
\end{teo}

To prove  Theorem \ref{main1},  we establish the existence of solutions for \eqref{obstacleepmfg1sys}-\eqref{obstacleepmfg2sys} in  Theorem \ref{main2} and prove $\epsilon$ independent bounds. The analysis
of \eqref{obstacleepmfg1sys}-\eqref{obstacleepmfg2sys}
begins in Section \ref{apest} where we examine various a priori estimates.
Next, in Section \ref{fapest}, we consider separately the two different nonlinearities,  $g(m)=\ln m$ and $g(m)=m^\alpha$.  In these two sections, our estimates are uniform
in $\epsilon$. 
In contrast, in Section \ref{lipbounds}, we prove $L^\infty$ estimates for
$\theta$ and Lipschitz bounds for $u$ that depend on $\epsilon$. These are
crucial 
in the proof of Theorem \ref{main2} that we present in Section \ref{pmain2}.
This proof combines the a priori estimates with the continuation method.
The paper ends with the proof of Theorem \ref{main1} in Section \ref{pmain1} and a brief discussion of convergence and uniqueness in Section \ref{secunlim}. 


\section{Main assumptions }
\label{opmfg}



We begin by discussing the assumptions on  $H^i$, $g,$ and $\psi$ used in the study of  \eqref{obstacleepmfg1}-\eqref{obstacleepmfg2}.
On the Hamiltonian, $H^i$, we assume standard hypotheses that hold in a large class of problems. In particular, they are satisfied 
by the example \eqref{HHYP}.  
To simplify the presentation, 
we select assumptions
 compatible with quadratic growth of the Hamiltonian, see Remark \ref{R2} below.  
 Regarding the dependence on the measure: for every coordinate, $i$, we have the same nonlinearity, $g$,  evaluated at the 
 coordinate $\theta^i$. Some of our estimates are valid without substantial changes in the corresponding proofs if $g$ is replaced by a function, $g^i$, 
 depending on all coordinates of $\theta$ or even on $x$. Naturally,  Assumption \ref{A2} must be modified in a suitable way. Finally, we work with positive switching costs, $\psi^{ij}$. The positivity  condition is natural in
 optimal switching because it prevents the occurrence of infinitely many switches. These conditions and the assumptions that follow are unlikely to give
 the most general case under which our techniques hold. Our choice reflects a balance between generality and
 simplicity of the proofs. 

\begin{hyp}
\label{A1} The Hamiltonian, $H^i$, the nonlinearity, $g,$ and the switching cost, $\psi^{ij}$, satisfy:
\begin{enumerate}
\item 
For $1\leq i\leq d$,
 $H^i:\Tt^N\times\R^N\to\R$ is $C^\infty$ and positive. 
\item 
 $g:\R^+\rightarrow\R$ is  $C^\infty$  and strictly increasing; that is,   $g'>0$.
 \item For $1\leq i, j\leq d$, the function 
    $\psi^{ij}:\Tt^N\to \Rr$ is of class
 $C^\infty(\Tt^N)$. Furthermore, for  $x\in \Tt^N$,     
     $\psi^{ij}(x)>0$. 
\end{enumerate}
\end{hyp}

As usual, we identify whenever convenient, functions in $\Tt^N$\ as $\Z^N$-periodic functions in  $R^N$.

\begin{hyp}
\label{A2} The function $g$ satisfies
the following.
\begin{enumerate}
\item For any $C_0>0,$ there exists $C_1$ such that 
\beq\label{gconvexass} \intor \theta g(\theta)dx\geq -C_1\eeq
for any non-negative $\theta\in L^1(\Tt^N)$ with $\intor\theta dx\leq C_0.$
\item There exists $C>0$ such that, for any $\theta>0$,
\beq\label{ggrowth}g(\theta)\leq \frac{1}{2}\theta g(\theta)+C.\eeq
\end{enumerate}
\end{hyp}

\begin{rem}
The functions $g(\theta)=\ln \theta$ and $g(\theta)=\theta^\alpha$, for $\alpha>0$, satisfy the preceding assumption. 
\end{rem}

\begin{hyp}
\label{A3}
There exist constants, $c, C\geq 0,$ such that 
\beq\label{dphp}
H^i(p,x)-D_pH^i(p,x)\cdot p\le -cH^i(p,x)+C
\eeq
for all $p\in\R^N$, $x\in \Tt^N$, and $1\leq i\leq d$. 
\end{hyp}

\begin{rem}
Consider the Lagrangian, $L^i$, associated with the Hamiltonian $H^i$ given by
 \[
L^i(x,v)=\sup_{p\in \Rr^N} -p\cdot v-H^i(p,x).  \]
Because the supremum is achieved for $v=-D_pH^i(p,x)$,
\[
L^i(x,v)=D_pH^i (p,x)\cdot p -H^i(p,x).
\]
Accordingly, the preceding hypothesis gives a lower bound on $L^i$. 
\end{rem}

\begin{hyp}
\label{A4}
There exists $\gamma>0$ such that 
\beq\label{Hconvexity} H^i_{p_kp_j}(p,x)\xi_k\xi_j\ge \gamma|\xi |^2\eeq
for all $x\in \Tt^N$ and $p,\,\xi\in\R^N$.

There exist   $C,\,c>0$ such that 
 \beq\label{growthH}\begin{split}
 &|D^2_{pp}H^i|\le C, \\ &|D^2_{xp}H^i|\le C(1+|p|),\\& |D^2_{xx}H^i|\le C(1+|p|^2).\end{split}
\eeq
\end{hyp}

\begin{rem}
\label{R2}
The preceding assumption implies that there exists $C>0$ such that 
\beq\label{Hquadratic}\frac{\gamma}{2}|p|^2 -C\le H^i(p,x)\le C |p|^2 +C,
\eeq and 
\beq\label{DpHsulinear}\begin{split}& |D_pH^i(p,x)|\le C(1+|p|),\\&
|D_xH^i(p,x)|\le C(1+|p|^2)\end{split}
\eeq
 for all $p\in \R^N$ and $x\in \Tt^N$. 
\end{rem}


\begin{hyp}
\label{A5}
There exist constants, $C,\widetilde{C}>0,$ and $\alpha\geq 0$ 
such that 
\beq
\label{g'prop}C\theta^{\al-1}\leq g'(\theta)\leq \widetilde{C}\theta^{\al-1}+\widetilde{C}\eeq
for any $\theta\ge 0$. Two specific cases of interest are
\begin{itemize}
\item[{\bf L} - ] $g(\theta)=\ln \theta$;

\item[{\bf P} - ] $g(\theta)=\theta^\alpha$, $0\leq\alpha \leq 1$. 

\end{itemize}

In the {\bf P} case, the additional constraint, $\alpha<\frac 2 N$, is denoted by {\bf P-$\frac 2 N$}. 

\end{hyp}

The next two assumptions are of a technical nature and are used in the study of the {\bf P-$\frac 2 N$}  case.
Assumption \ref{A6} is employed in Proposition \ref{MEst} 
 to obtain a lower bound for $\theta^i$. Assumption \ref{A7} is fundamental in the proof of Proposition \ref{liplem}.

\begin{hyp}
\label{A6}
For $1\leq i\leq d$, we have
\beq\label{bgradcondpower}D_{px}^2H^i(0,x)=D_xH^i(0,x)D_pH^i(0,x)=0 \quad\text{for any }x\in\Tt^N \eeq
and
\beq\label{blipconprop} \max_{x\in\Tt^N}H^i(0,x)<1.\eeq
\end{hyp}

\begin{rem}
The preceding assumption is used to prove lower bounds for $\theta^i$. Because $H^i\geq 0$, the bound \eqref{blipconprop} gives an oscillation condition for $H^i(0,x)$. This oscillation condition is natural in light of the example considered in \cite{GPV}, Chapter 3.  In that reference and also in \cite{GNP}, various examples of first-order MFGs are  shown to have a vanishing density. The oscillation of $H(0,x)$ plays an essential role
in these examples. 
\end{rem}

\begin{rem}
The number $1$ on the right-hand side of \eqref{blipconprop} corresponds to the source term in the Fokker-Planck equation \eqref{obstacleepmfg2sys}. 
Suppose that we modify   \eqref{obstacleepmfg2sys} and consider a source,
$\upsilon>0$; that is,
\[
-\di (D_p H^i(Du^i,x)\theta^i)+\theta^i+\sum_{j\ne i}\beta'_\ep(u^i-u^j-\psi^{ij}) \theta^i -\beta'_\ep(u^j-u^i-\psi^{ji}) \theta^j=\upsilon.  
\]
Then, \eqref{blipconprop} becomes
\[
\max_{x\in\Tt^N}H^i(0,x)<\upsilon. 
\]
\end{rem}

\begin{hyp}
\label{A7}
The value $\alpha$ in Assumption \ref{A5}  satisfies $\alpha\in[0,\alpha_0)$, where $\alpha_0$ solves\begin{equation}
\label{lipsalphaassmp}
2\alpha_0=(\alpha_0+1)\beta (\beta-1), \qquad \text{with} \quad \beta=\sqrt{\frac{2^*}{2}},
\end{equation}
if $N>2$, and $\alpha_0=\infty$ if $N\leq 2$.
\end{hyp}

\begin{rem}
In the $N\leq 3$ case, the value $\alpha_0$ determined by \eqref{lipsalphaassmp} is larger than $\frac 2 N$. Whereas, if $N>3$ the opposite inequality holds. 
\end{rem}



Our last assumption is required in the study of the penalized problem. For $\epsilon>0$, we choose a penalty, $\beta_\ep,$ satisfying the following assumption.

\begin{hyp}
\label{B1}
$\beta_\ep:\R\rightarrow\R$, smooth, with  $\beta_\ep'\geq 0$, 
$\beta_\ep''\geq 0$  with 
\beq\label{beta1}\beta_\ep(s)=0\quad \text{for } s\le 0,\quad\beta_\ep'(s)\leq C\beta_\epsilon''(s)\quad \text{for }s> 0,\eeq
and
$\beta_\epsilon(s)\to \infty$ as $\epsilon\to 0$  for $s>0$.
\end{hyp}

\begin{rem}
\label{whatever}
From the preceding assumption, we get 
 \beq\label{beta2}
\beta_\ep(s)-s\beta_\ep'(s)\leq 0\quad \text{for }s\in\R.\eeq
\end{rem}

The preceding assumption is standard in the setting of variational inequalities and optimal switching. In the context of MFGs,
a similar penalty was used in \cite{GPat} to study the obstacle problem.


\section{A priori estimates}
\label{apest}

Here, we prove a priori estimates for classical solutions of \eqref{obstacleepmfg1sys}-\eqref{obstacleepmfg2sys}. The purpose of these estimates is twofold: first, to obtain the existence of solutions; 
second, to take the limit $\epsilon\to 0$. For that, we seek to prove bounds that are uniform in $\epsilon$. 
We begin with a simple consequence of the maximum principle for weakly coupled systems.
\begin{pro}
        Suppose that Assumptions \ref{A1} and \ref{B1}  hold. 
        Let $(u,\theta)$ be a $C^\infty$ solution of \eqref{obstacleepmfg1sys}-\eqref{obstacleepmfg2sys}.
        Then,  for $i=1,\ldots,d$, we have $\theta^i\geq 0$.
\end{pro}
\begin{proof}
The proof of this Lemma is a straightforward  application of the maximum principle to \eqref{obstacleepmfg2sys}, see \cite{CGT2} for a similar proof. 
\end{proof}

As is standard in MFG problems, we can get several  estimates  by multiplying \eqref{obstacleepmfg1sys}-\eqref{obstacleepmfg2sys}
by $1$, $\theta^i$ or $u^i$, adding or subtracting, and integrating by parts. We record these in the next lemma. 

\begin{lem}\label{estimate1lem}
        Suppose that Assumptions \ref{A1}-\ref{A3}  and \ref{B1} hold. 
        Let $(u,\theta)$ be a $C^\infty$ solution of \eqref{obstacleepmfg1sys}-\eqref{obstacleepmfg2sys}.
        Then,  there exists a constant, $C,$ that does not depend on the particular solution nor on  $\epsilon$, such that,
        for $i=1\hdots d$,
\begin{equation}
\label{btl1}
0\leq \intor \theta^idx\leq C, 
\end{equation}
\begin{equation}
\label{btl2}
\left|\intor \theta^ig(\theta^i)dx\right| \leq C,
\end{equation}
\begin{equation}
\label{btl3}
\left|\intor u^idx\right|\leq C,
\end{equation}
\begin{equation}
\label{btl4}
\intor |Du^i|^2 \theta^i dx\leq C, 
\end{equation}
\begin{equation}
\label{btl45}
\sum_{i,j=1, i\neq j}^d \intor \beta'_\ep(u^i-u^j-\psi^{ij})  \psi^{ij}\theta^idx \leq C, 
\end{equation}
and
\begin{equation}
\label{btl5}
\intor  |Du^i|^2 dx\leq C.
\end{equation}
\end{lem}
\begin{proof}
By 
summing over $i$ the equations in \eqref{obstacleepmfg2sys}, 
we gather that
$$
\sum_{i=1}^d-\di (D_p H^i(Du^i,x)\theta^i)+\theta^i=d.
$$
Hence, integrating on $\Tt^N$, we get
\[
 \intor \theta^idx\le\sum_{i=1}^d\intor \theta^idx=d 
\]
 for any  $i=1,...,d$. Thus, \eqref{btl1} holds. 
Due to Assumptions \ref{A1} and \ref{B1},    $H^i$ and $\beta_\ep$  are non-negative. Consequently, 
 we infer that
 \begin{equation}
 \label{XYZ}
 \intor u^idx\leq  \intor g(\theta^i)dx.
 \end{equation}
 Next, we multiply \eqref{obstacleepmfg1sys} by $\theta^i$, sum over $i,$ and integrate. Accordingly,  we gather
the identity \beq\label{firsteqthetaint}
 \begin{split}
& \sum_{i=1}^d\intor H^i(Du^i,x)\theta^i+ u^i\theta^i+\sum_{j\ne i}\beta_\ep(u^i-u^j-\psi^{ij})\theta^idx\\&\qquad= \sum_{i=1}^d\intor \theta^ig(\theta^i)dx.
 \end{split}
 \eeq
 Next, we  multiply \eqref{obstacleepmfg2sys} by $u^i$, add over $i,$ and integrate by parts
to conclude that  \beq\label{secondequiinteg}
 \begin{split}
 &\sum_{i=1}^d\intor \Big[ D_pH^i(Du^i,x)\cdot Du^i\theta^i+ u^i\theta^i\\&\quad +\sum_{j\ne i}\beta'_\ep(u^i-u^j-\psi^{ij}) \theta^i u^i
 -\beta'_\ep(u^j-u^i-\psi^{ji}) \theta^j u^i\Big]dx\\&\qquad =\sum_{i=1}^d\intor u^idx.\end{split}
 \eeq
 Subtracting equations \eqref{firsteqthetaint} and \eqref{secondequiinteg}, we get
  \beqs
  \begin{split}& \sum_{i=1}^d\intor \theta^ig(\theta^i)dx\\&
  =\sum_{i=1}^d\intor H^i(Du^i,x)\theta^i+ u^i\theta^i+\sum_{j\ne i}\beta_\ep(u^i-u^j-\psi^{ij})\theta^idx\\&
  =\sum_{i=1}^d\intor (H^i(Du^i,x)-D_p H^i(Du^i,x)\cdot Du^i)\theta^i+u^idx\\&
  +\sum_{i,j=1, i\neq j}^d\intor\beta_\ep(u^i-u^j-\psi^{ij})\theta^idx\\&
+  \sum_{i,j=1, i\neq j}^d\intor -\beta'_\ep(u^i-u^j-\psi^{ij}) \theta^i u^i+\beta'_\ep(u^j-u^i-\psi^{ji}) \theta^j u^idx.
 \end{split}
 \eeqs
According to Assumption \ref{A3}, we have
 \beqs
  \begin{split}
 &\sum_{i=1}^d\intor (H^i(Du^i,x)-D_pH^i(Du^i,x)\cdot Du^i)\theta^i dx\\&\quad \leq -c  \sum_{i=1}^d\intor H^i(Du^i,x)\theta^idx + C 
 \end{split}
 \eeqs
 using \eqref{btl1}.
 Moreover, we have
 \beqs
  \begin{split}
&\sum_{i,j=1, i\neq j}^d\intor\beta_\ep(u^i-u^j-\psi^{ij})\theta^idx\\&
 \quad+ \sum_{i,j=1, i\neq j}^d\intor -\beta'_\ep(u^i-u^j-\psi^{ij}) \theta^i u^i+\beta'_\ep(u^j-u^i-\psi^{ji}) \theta^j u^idx\\&
 =\sum_{i,j=1, i\neq j}^d \intor[\beta_\ep(u^i-u^j-\psi^{ij})-\beta'_\ep(u^i-u^j-\psi^{ij}) (u^i-u^j-\psi^{ij})]\theta^idx\\
 &\quad -\sum_{i,j=1, i\neq j}^d \intor \beta'_\ep(u^i-u^j-\psi^{ij})  \psi^{ij}\theta^idx\\&
 \leq -\sum_{i,j=1, i\neq j}^d \intor \beta'_\ep(u^i-u^j-\psi^{ij})  \psi^{ij}\theta^idx
 \end{split}
 \eeqs
by \eqref{beta2} in Remark \ref{whatever}. 
Gathering the previous estimates, we conclude
that   \beq
  \label{GE}
  \begin{split}
&\sum_{i=1}^d\intor \theta^ig(\theta^i)+c  H^i(Du^i,x)\theta^idx+\sum_{i,j=1, i\neq j}^d \intor \beta'_\ep(u^i-u^j-\psi^{ij})  \psi^{ij}\theta^idx\\
&\quad \leq \sum_{i=1}^d \intor u^idx+C\leq \sum_{i=1}^d\intor g(\theta^i)dx + C 
\end{split}
 \eeq
 using \eqref{XYZ}.
Then, Assumption \ref{A2} implies 
$$\intor \theta^ig(\theta^i)dx\leq C.$$
On the other hand, \eqref{gconvexass} in Assumption \ref{A2} and \eqref{btl1} give
$$\intor \theta^ig(\theta^i)dx\ge -C.$$ 
Therefore,
 \eqref{btl2} holds. 
Using \eqref{btl2} and the bound \eqref{ggrowth} from Assumption \ref{A2}, we get  \eqref{btl3}. In addition,  for any  $i=1,...,d$,
\beqs \intor H^i(Du^i,x)\theta^idx\le C.
\eeqs 
The last estimate combined with \eqref{Hconvexity} implies \eqref{btl4}.  A similar argument yields  \eqref{btl45}.

Finally,
the bound \eqref{btl5} follows
from   \eqref{obstacleepmfg1sys}  by combining  \eqref{Hconvexity}, the non-negativity of $\beta_\ep$,  and the previous results
with the estimate
\beqs
\intor  |Du^i|^2 dx\leq C+\intor (g(\theta^i)-u^i)dx\leq C.
\eeqs
\end{proof}

\begin{lem}\label{estimateslem2}
        Suppose that Assumptions \ref{A1}-\ref{A5}  and \ref{B1} hold. 
        Let $(u,\theta)$ be a $C^\infty$ solution of \eqref{obstacleepmfg1sys}-\eqref{obstacleepmfg2sys}.
        Then,  there exists a constant, $C,$ that does not depend on the particular solution nor  on $\epsilon$, such that,  for $i=1,\ldots, d$,
\beq\label{d2utheta}\intor |D^2u^i|^2\theta_idx\leq C,\eeq
\beq\label{g'thetadth}\intor g'(\theta^i)|D\theta^i|^2dx\leq C,\eeq and 
\beq
\label{w12alpthet}\|(\theta^i)^\frac{\alpha+1}{2}\|_{W^{1,2}(\Tt^N)}\leq C.
\eeq    
\end{lem}
\begin{proof}
We begin by 
differentiating \eqref{obstacleepmfg1sys} twice with respect to $x_k$ and then summing over $k$. In this way, we get  
\beqs\begin{split}&D_pH^i\cdot D( \Delta u^i)+ H^i_{x_kx_k}+ 2 H^i_{x_kp_l}u^i_{x_lx_k}+H^i_{p_lp_m}u^i_{x_lx_k}u^i_{x_mx_k}+\Delta u^i\\&
+\sum_{j\ne i}\beta_\ep'(u^i-u^j-\psi^{ij})\Delta (u^i-u^j-\psi^{ij})+\beta_\ep''(u^i-u^j-\psi^{ij})|D(u^i-u^j-\psi^{ij})|^2\\
&\qquad =\Delta(g(\theta^i)).
\end{split}\eeqs
Next, 
we multiply the previous equation by $\theta^i$,  add in the index $i,$ and integrate by parts to  conclude that 
\beq
\label{pieq}
\begin{split}\sum_{i=1}^d\intor &\Delta(g(\theta^i))\theta^idx\\
=&\sum_{i=1}^d\intor (D_pH^i\cdot D(\Delta u^i)+\Delta u^i+ \sum_{j\ne i}\beta_\ep'(u^i-u^j-\psi^{ij})\Delta u^i)\theta^idx\\&
- \sum_{i=1}^d\sum_{j\ne i}\intor\beta_\ep'(u^i-u^j-\psi^{ij})\Delta (u^j+\psi^{ij})\theta^idx\\&
+\intor(  H^i_{x_kx_k}+ 2 H^i_{x_kp_l}u^i_{x_lx_k}+H^i_{p_lp_m}u^i_{x_lx_k}u^i_{x_mx_k})\theta^idx\\&
+\sum_{i=1}^d\sum_{j\ne i}\intor \beta_\ep''(u^i-u^j-\psi^{ij})|D(u^i-u^j-\psi^{ij})|^2\theta^idx.
\end{split}
\eeq
Multiplying \eqref{obstacleepmfg2sys} by $\Delta u^i$ and integrating by parts results in 
\beqs \begin{split}&\sum_{i=1}^d\intor (D_pH^i\cdot D(\Delta u^i)+\Delta u^i+ \sum_{j\ne i}\beta_\ep'(u^i-u^j-\psi^{ij})\Delta u^i)\theta^idx\\&
=\sum_{i=1}^d\sum_{j\ne i}\intor(\beta'_\ep(u^j-u^i-\psi^{ji}) \Delta u^i\theta^j+\Delta u^idx\\&
=\sum_{i=1}^d\sum_{j\ne i}\intor(\beta'_\ep(u^i-u^j-\psi^{ij}) \Delta u^j\theta^i dx. 
\end{split}
\eeqs
Using the previous identity in \eqref{pieq} gives
\beqs\begin{split}
\sum_{i=1}^d\intor \Delta(g(\theta^i))\theta^idx&
=- \sum_{i=1}^d\sum_{j\ne i}\intor\beta_\ep'(u^i-u^j-\psi^{ij})\Delta\psi^{ij}\theta^idx\\&
+\intor(H^i_{x_kx_k}+ 2 H^i_{x_kp_l}u^i_{x_lx_k}+H^i_{p_lp_m}u^i_{x_lx_k}u^i_{x_mx_k})\theta^idx\\&
+\sum_{i=1}^d\sum_{j\ne i}\intor \beta_\ep''(u^i-u^j-\psi^{ij})|D(u^i-u^j-\psi^{ij})|^2\theta^idx.\\&
\end{split}
\eeqs
Taking into account that $\Delta\psi^{ij}$
is bounded and $\psi^{ij}>0$, estimate \eqref{btl45} implies that
\[
\left|
\sum_{i, j=1, i\neq j}^d\intor\beta_\ep'(u^i-u^j-\psi^{ij})\Delta\psi^{ij}\theta^idx
\right|\leq C.
\]
Because  $\beta_\ep''\geq 0$ and because of the uniform convexity from \eqref{Hconvexity},  \eqref{growthH} and \eqref{btl4},
we obtain
 \begin{equation}
         \label{see}
         \begin{split}
        & \sum_{i=1}^d\intor g'(\theta^i)|D\theta^i|^2dx+C\sum_{i=1}^d\intor |D^2u^i|^2\theta^idx\\&\qquad \leq C \sum_{i=1}^d\intor (1+|Du^i|^2)\theta^idx\leq C.
         \end{split}
                 \end{equation}
Hence, we have \eqref{d2utheta} and \eqref{g'thetadth}.
Moreover, from \eqref{g'thetadth} and  \eqref{g'prop}, we infer that 
 \beqs\intor(\theta^i)^{\alpha-1}|D\theta^i|^2dx\leq C\intor g'(\theta^i)|D\theta^i|^2dx\leq C;\eeqs
 that is, $|D(\theta^i)^\frac{\alpha+1}{2}|\in L^2(\Tt^N)$.  By \eqref{btl1},  $\theta^i\in L^1(\Tt^N)$. Thus, the previous inequality and the Poincar\'{e} inequality imply \eqref{w12alpthet}.
\end{proof} 


\section{Further a priori estimates}
\label{fapest}

In this section, we prove additional a priori estimates for logarithmic
(Assumption \ref{A5} {\bf L}) 
 and power-like nonlinearities (Assumptions \ref{A5} {\bf P} and {\bf P-$\frac 2 N$}).
These two cases are examined separately.  
 Nevertheless, for both
  the logarithmic nonlinearity and 
 for the power case, if $\alpha<\frac 2 N$   (Assumption \ref{A5} {\bf P-$\frac 2 N$}), 
 we obtain
 similar, 
 $\ep$-independent lower bounds
 on $\theta$, on  $\|u\|_{W^{2,2}(\Tt^N)}$, and on $\|u\|_{W^{1,p}(\Tt^N)}$ for any $p\geq1$.

\subsection{The logarithmic case}

Here, we consider the logarithmic nonlinearity $g(\theta)=\ln \theta$. 

\begin{pro}
\label{prolog}  
        Suppose that Assumptions  \ref{A1}-\ref{A4},  \ref{A5} {\bf L}  and \ref{B1} hold.
        Let $(u,\theta)$ be a $C^\infty$ solution of \eqref{obstacleepmfg1sys}-\eqref{obstacleepmfg2sys}.
        Then, there exist constants  $C,\,C_p,\,\theta_0>0$ that do not depend on the particular solution nor on $\epsilon$, such that, for $i=1,\ldots,d$
and for any
$p\in[1,+\infty)$, 
                \beq\label{w1pestlog}\|u^i\|_{W^{1,p}(\Tt^N)}\leq C_p.\eeq
                Moreover, for any
                $\gamma\in(0,1)$,
                \beq\label{calphaestlog}\|u^i\|_{C^\gamma(\Tt^N)}\leq C. \eeq
In addition, 
        \beq\label{lowerboundthelog}\theta^i\geq\theta_0\quad\text{in }\Tt^N,\eeq
        \beq\label{w12thetalog}\|\theta^i\|_{W^{1,2}(\Tt^N)}\leq C,\eeq
        and
        \beq\label{w22estlog}\|u^i\|_{W^{2,2}(\Tt^N)}\leq C.\eeq          
        \end{pro}
\begin{proof}
In what follows, we use $C$ and $C_p$ to denote any of several constants, possibly depending on $p$ but independent of $\epsilon$.
We remark that, for any $p\geq 1$, there exists a constant, $C_p>0,$ such that 
$$\log(\theta^i)\leq (\theta^i)^\frac{1}{p}+C_p.$$ Therefore, from \eqref{obstacleepmfg1sys}, using \eqref{Hquadratic} in Remark \ref{R2} and the positivity of  $\beta_\ep$, we infer that 
\beqs C|D u^i|^2\leq  (\theta^i)^\frac{1}{p}+C_p-u^i=(\theta^i)^\frac{1}{p}+C_p-\left(u^i-\intor u^idx\right)-\intor u^i dx.\eeqs
Combining the previous inequality with  \eqref{btl3} yields
\beqs |D u^i|^{2p}\leq  C\theta^i+C\left|u^i-\intor u^idx\right|^p+C_p.\eeqs
Then, integrating,  using \eqref{btl1} and the Poincar\'{e} inequality, we get 
\beqs\begin{split} \intor |D u^i|^{2p}dx & \leq  C\intor \theta^idx+C\intor\left|u^i-\intor u^idx\right|^pdx +C_p\\&
\leq C_p\intor |D u^i|^{p}dx+C_p\\&
\leq \frac{1}{2}\intor |D u^i|^{2p}dx+C_p.
\end{split}\eeqs
 We conclude that, for any $p\geq 1$,
 \beqs \intor|D u^i|^{2p}dx\leq  C.\eeqs This bound, together with \eqref{btl3} and the Poincar\'{e} inequality, gives \eqref{w1pestlog}. The Sobolev Embedding Theorem then implies 
 \eqref{calphaestlog}.   In particular, we have 
 \beqs \|u^i\|_{L^\infty(\Tt^N)}\leq C.\eeqs From  \eqref{obstacleepmfg1sys}, the 
 previous estimate and the positivity of $H$ and $\beta$, we infer that 
 $$\log(\theta^i)\geq -C,$$ from which \eqref{lowerboundthelog} follows.

Estimate \eqref{w12thetalog} is a consequence of \eqref{btl1}, \eqref{g'thetadth}, \eqref{lowerboundthelog} and the Poincar\'{e} inequality.
 
  Finally, estimate \eqref{w22estlog} is a consequence of \eqref{btl3}, \eqref{btl5}, \eqref{d2utheta}, \eqref{lowerboundthelog} and the Poincar\'{e} inequality.
 
\end{proof}

\subsection{Power case}

We devote this section to the study of power nonlinearities. We begin by examining the general case, Assumption \ref{A5} {\bf P}. Then,
we obtain additional results
by considering 
 Assumption \ref{A5} {\bf P-$\frac 2 N$}. As in the previous section, our estimates are 
uniform in $\epsilon$. 

\begin{pro}
\label{propower}
        Suppose that Assumptions  \ref{A1}-\ref{A4}, \ref{A5} {\bf P}  and \ref{B1} hold.
        Let $(u,\theta)$ be a $C^\infty$ solution of \eqref{obstacleepmfg1sys}-\eqref{obstacleepmfg2sys}.
        Then, there exist constants,  $C>0$ and $\gamma\in(0,1),$ that do not depend on the particular solution nor on $\epsilon$, such that, for $i=1,\ldots,d,$
        \beq\label{w1pestpower}\|u^i\|_{W^{1,\frac{2}{\alpha}}(\Tt^N)}\leq C\eeq 
         and 
        \beq\label{d2uthetafurterpower}\intor|D((\theta^i)^\frac{\alpha+1}{2}Du^i)|^\frac{2}{\alpha+1}dx\leq C.\eeq
        If, in addition, Assumption \ref{A5} {\bf P-$\frac 2 N$} holds, then there exists $\gamma=\gamma(\alpha)$ such that 
        \beq\label{holderestimpower} \|u^i\|_{C^\gamma(\Tt^N)}\leq C.\eeq
        \end{pro}
\begin{proof}
In what follows, we denote by $C$  several constants that are independent of $\epsilon$ and $\delta$. 
 From \eqref{obstacleepmfg1sys},  \eqref{Hquadratic}, and $\beta_\ep\geq 0$, we infer that
 \beqs C|D u^i|^2\leq (\theta^i)^\alpha+C-u^i= (\theta^i)^\alpha+C-\left(u^i-\intor u^idx\right)-\intor u^i dx.\eeqs
 Consequently,  from  \eqref{btl3},
 \beqs |D u^i|^\frac{2}{\alpha}\leq  C\theta^i+C\left|u^i-\intor u^idx\right|^\frac{1}{\alpha}+C.\eeqs
 Then, integrating and using \eqref{btl1} and the Poincar\'{e} inequality, we get 
 \beqs\begin{split} \intor |D u^i|^\frac{2}{\alpha}dx&\leq C\intor \theta^idx+C\intor\left|u^i-\intor u^idx\right|^\frac{1}{\alpha} dx+C\\&
\leq C\intor |D u^i|^\frac{1}{\alpha}dx+C\\&
\leq \frac{1}{2}\intor |D u^i|^\frac{2}{\alpha}dx+C.
\end{split}\eeqs
Therefore, 
$$\intor |D u^i|^\frac{2}{\alpha}dx\leq C,$$ which gives,  together with \eqref{btl3} and the Poincar\'{e} inequality,  the bound \eqref{w1pestpower}.
If  $\alpha\in \left(0, \frac{2}{N}\right)$, then  $\frac{2}{\alpha}>N$. Therefore,
estimate \eqref{holderestimpower} is a consequence of \eqref{w1pestpower} combined with the  Sobolev Embedding Theorem.

Next,  to prove \eqref{d2uthetafurterpower}, we compute
\beq\label{dthetducomp} D((\theta^i)^\frac{\alpha+1}{2}Du^i)=\frac{\alpha+1}{2}\theta^\frac{\alpha-1}{2}Du^i\otimes D\theta^i+\theta^\frac{\alpha+1}{2}D^2 u^i.\eeq
Now, using  H\"{o}lder's inequality, we have
\beqs\begin{split}&\intor \left[\theta^\frac{\alpha-1}{2}|D\theta^i||Du^i|\right]^\frac{2}{\alpha+1}dx\\&
=\intor(\theta^i)^\frac{\alpha-1}{\alpha+1}|D\theta^i|^\frac{2}{\alpha+1}|Du^i|^\frac{2}{\alpha+1}dx\\&
\leq \left[\intor\left[(\theta^i)^\frac{\alpha-1}{\alpha+1}|D\theta^i|^\frac{2}{\alpha+1}\right]^{\alpha+1}dx\right]^\frac{1}{\alpha+1}\left[\intor|Du^i|^{\frac{2}{\alpha+1}(\alpha+1)'}dx\right]^\frac{1}{(\alpha+1)'}\\&
=\left[\intor(\theta^i)^{\alpha-1}|D\theta^i|^2dx\right]^\frac{1}{\alpha+1}\left[ \intor|Du^i|^\frac{2}{\alpha}dx\right]^\frac{\alpha}{\alpha+1}.
\end{split}\eeqs
From \eqref{w12alpthet} and \eqref{w1pestpower}, we infer that 
\beq\label{dthetduest1} \intor \left[\theta^\frac{\alpha-1}{2}|D\theta^i||Du^i|\right]^\frac{2}{\alpha+1}dx\leq C.\eeq
Next,  using H\"{o}lder's inequality again, we have
\beqs\begin{split} \intor \left[|D^2u^i|(\theta^i)^\frac{\alpha+1}{2}\right]^\frac{2}{\alpha+1}dx&= \intor |D^2u^i|^\frac{2}{\alpha+1}\theta^idx
=\intor\left[ |D^2u^i|^2\theta^i\right]^\frac{1}{\alpha+1}(\theta^i)^\frac{\alpha}{\alpha+1}dx\\&
\leq \left[\intor|D^2u^i|^2\theta^idx\right]^\frac{1}{\alpha+1}\left[\intor \theta^idx\right]^\frac{\alpha}{\alpha+1}.
\end{split}\eeqs
From \eqref{btl1} and \eqref{d2utheta}, we gather the bound
\beq\label{dthetduest2}\intor \left[|D^2u^i|(\theta^i)^\frac{\alpha+1}{2}\right]^\frac{2}{\alpha+1}dx\leq C.\eeq
Estimate \eqref{d2uthetafurterpower} is then a consequence of \eqref{dthetducomp}, \eqref{dthetduest1}  and  \eqref{dthetduest2}.
\end{proof}

\begin{pro}
        \label{MEst}
        Suppose that Assumptions  \ref{A1}-\ref{A4},  \ref{A5} {\bf P-$\frac 2 N$}, \ref{A6}, and \ref{B1} hold. 
        Let $(u,\theta)$ be a $C^\infty$ solution of \eqref{obstacleepmfg1sys}-\eqref{obstacleepmfg2sys}. Then,  for  $i=1,\ldots,d$ and any $x\in \Tt^N$, we have
        \beq\label{theta>1power}\theta^i(x)\geq \left(1-\max_{x\in\Tt^N\atop j=1,\ldots,d}H^j(0,x)\right)^{\frac 1 \alpha}.
 \eeq
Moreover, there exists $C>0$ that does not depend on the particular solution nor on $\epsilon$, such that
        \beq\label{w12thetapower}\|\theta^i\|_{W^{1,2}(\Tt^N)}\leq C\eeq and
        \beq\label{w22estpower}\|u^i\|_{W^{2,2}(\Tt^N)}\leq C.\eeq
        \end{pro}
\begin{proof}
We begin the proof by establishing a lower bound on $u$. 
Let $i\in\{1,\ldots,d\}$ and $x_0\in\Tt^N$ be such that
\beqs u^{i}(x_0)=\min_{j=1,\ldots,d\atop x\in\Tt^N}u^j(x).\eeqs
Then, we have 
\beq\label{derix_0}Du^i(x_0)=0,\,D^2u^i(x_0)\geq 0,\eeq 
and 
$$
u^i(x_0)\leq u^j(x_0)\quad\text{for any }j=1,\ldots,d.
$$
In particular, the last inequality implies 
\beq\label{betax_0}\beta_\ep(u^i(x_0)-u^j(x_0)-\psi^{ij}(x_0))=\beta_\ep'(u^i(x_0)-u^j(x_0)-\psi^{ij}(x_0))=0 \quad\text{for any }j=1,\ldots,d.\eeq
From  \eqref{obstacleepmfg1sys}, \eqref{derix_0}, and \eqref{betax_0}, we infer that 
\beq\label{u(x0)value} H^i(0,x_0)+u^i(x_0)=g(\theta^i(x_0)).\eeq
We can substitute 
$$\theta^i=g^{-1}\left(H^i(Du^i,x)+u^i+\sum_{j\ne i}\beta_\ep(u^i-u^j-\psi^{ij})\right)$$
in \eqref{obstacleepmfg2sys} to get
\beqs\begin{split}&-\theta^i H^i_{p_kp_j}u^i_{x_jx_k}-\theta^i H^i_{p_kx_k}-\frac{1}{g'(\theta^i)}\big[H^i_{p_k}H^i_{p_j}u^i_{x_jx_k}+H^i_{x_k}H^i_{p_k}+H^i_{p_k}u^i_{x_k}\\&+
\sum_{j\neq i}\beta_\ep'(u^i-u^j-\psi^{ij})(u^i-u^j-\psi^{ij})_{x_k}H^i_{p_k}\big]+\theta^i+\sum_{j\neq i}\beta_\ep'(u^i-u^j-\psi^{ij})\theta^i\\&
=\sum_{j\neq i}\beta_\ep'(u^j-u^i-\psi^{ji})\theta^j+1.
\end{split}\eeqs
Evaluating at $x=x_0$ and using \eqref{Hconvexity}, \eqref{derix_0} and \eqref{betax_0}, we obtain
\beqs -\theta^i H^i_{p_kx_k}(0,x_0)-\frac{1}{g'(\theta^i(x_0))}H^i_{x_k}(0,x_0)H^i_{p_k}(0,x_0)+\theta^i(x_0)\geq 1.\eeqs
Since  $H^i$  satisfies \eqref{bgradcondpower}, the preceding inequality can be rewritten as
\beqs \theta^i(x_0)\geq 1.\eeqs
Then, \eqref{u(x0)value}  and the last estimate imply
\beq\label{u(x0)valuebis} u^i(x_0)\geq g^{-1}\left(1-H^i(0,x_0)\right).\eeq
Now, from \eqref{obstacleepmfg1sys}, \eqref{u(x0)valuebis}, \eqref{bgradcondpower}, \eqref{blipconprop} and the positivity of $H^j$ and $\beta_\ep$, we infer that, for
any $x\in \Tt^N$ and  $j=1,\ldots,d$,
\beqs\begin{split}g(\theta^j(x))\geq H^j(Du^j,x)+u^j(x)\geq  u^i(x_0)\geq 1-H^i(0,x_0). \end{split}\eeqs Thus, \eqref{theta>1power} follows from the preceding inequality. 

Estimate \eqref{w12thetapower} follows by combining \eqref{btl1}, \eqref{g'thetadth}, \eqref{theta>1power} and the Poincar\'{e} inequality.
Finally, \eqref{w22estpower} is a consequence of \eqref{btl3}, \eqref{btl5}, \eqref{d2utheta}, \eqref{theta>1power}, and the Poincar\'{e} inequality. 
\end{proof}

\section{Lipschitz bounds}
\label{lipbounds}

In this section, we prove the Lipschitz continuity of $u$ for any solution $(u, \theta)$ of \eqref{obstacleepmfg1sys}-\eqref{obstacleepmfg2sys}. These bounds are used to establish
the existence of solutions by the continuation method. In contrast to the results in the preceding sections, the estimates here depend on $\epsilon$ and
are not valid for \eqref{obstacleepmfg1}-\eqref{obstacleepmfg2}.

\begin{lem}\label{liplem}
        Suppose that Assumptions \ref{A1}-\ref{A5}  and \ref{B1} hold, and that either
\begin{itemize}
        \item[-]  Assumption \ref{A5} {\bf L} or
        \item[-]  Assumptions \ref{A5} {\bf P-$\frac 2 N$}, \ref{A6} and \ref{A7} 
\end{itemize}   
        hold. Let $(u,\theta)$ be a $C^\infty$ solution of \eqref{obstacleepmfg1sys}-\eqref{obstacleepmfg2sys}.
        Then, there exists $C_\ep>0$ that does not depend on the particular solution, such that, for any $i=1,\ldots,d$, 
        \beq\label{thetabounded}\|\theta^i\|_{L^\infty(\Tt^N)}\leq C_\ep\eeq and
        \beq
        \label{lipepest}\|u^i\|_{W^{1,\infty}(\Tt^N)}\leq C_\ep.
        \eeq
\end{lem}
\begin{proof}
First, note that \eqref{lipepest} is an immediate consequence of \eqref{thetabounded}. Indeed, by combining  \eqref{thetabounded}  with   \eqref{Hquadratic}, by the positivity of $\beta_\ep,$ by the boundedness of $u^i$ (c.f. Proposition \ref{prolog} and Proposition \ref{propower}), and by   \eqref{obstacleepmfg1sys},
we get
$$C|D u^i|^2\leq H^i(D u^i)+C\leq  g(\theta^i)-u^i+C\leq C.$$
Consequently, we only need to prove   \eqref{thetabounded}. 

If Assumption \ref{A5} {\bf L} holds  or if  Assumptions \ref{A5} {\bf P-$\frac 2 N$} and \ref{A6} hold
by, respectively, Propositions
\ref{prolog} and \ref{MEst}, 
then there exists $\theta_0>0$, 
 such that, for any $i=1,\ldots,d$ and for any $x\in \Tt^N,$ \beq\label{thelowboundepasslem}\theta^i(x)\geq \theta_0>0.
 \eeq

For any fixed $\ep$,  by \eqref{calphaestlog} and \eqref{holderestimpower},  there exists a constant, $C,$ depending on $\ep$,  such that, for any  $i,j=1,\ldots,d$,
\beq\label{beta'bounded}\beta_\ep'(u^i-u^j-\psi^{ij})\leq C.\eeq

To prove \eqref{thetabounded}, we use a technique introduced in   \cite{E1}  and used in \cite{GPat} to study a mean-field-game obstacle problem. For $p>0$, we multiply the equation \eqref{obstacleepmfg2sys} by $\di ((\theta^i)^p D_pH^i(Du^i,x))$ and integrate by parts. Accordingly, we get\beq\label{lemlipine1}\begin{split} &\intor [\theta^i+\sum_{j\neq i}\beta_\ep'(u^i-u^j-\psi^{ij})\theta^i-\sum_{j\neq i}\beta_\ep'(u^j-u^i-\psi^{ji})\theta^j]\di ((\theta^i)^p D_pH^i)dx\\&
=\intor (\theta^i H^i_{p_k})_{x_k}  ((\theta^i)^p H^i_{p_j})_{x_j}dx \\&
=\intor (\theta^i H^i_{p_k})_{x_j}  ((\theta^i)^p H^i_{p_j})_{x_k}dx \\&
=\intor (\theta^i (H^i_{p_k})_{x_j}+\theta^i_{x_j}H^i_{p_k})((\theta^i)^p (H^i_{p_j})_{x_k}+p(\theta^i)^{p-1}\theta^i_{x_k}H^i_{p_j})dx\\&
=\intor (\theta^i)^{p+1}(H^i_{p_k})_{x_j}(H^i_{p_j})_{x_k}+p(\theta^i)^{p-1}H^i_{p_k}\theta^i_{x_k}H^i_{p_j}\theta^i_{x_j}\\&+(p+1)(\theta^i)^p\theta^i_{x_k}H^i_{p_j}(H^i_{p_k})_{x_j}dx\\&
=:\intor I_1+I_2+I_3 dx.
\end{split}\eeq
Using   \eqref{growthH} in Assumption \ref{A4}, we get
\beqs\begin{split}
I_1&=(\theta^i)^{p+1}(H^i_{p_kp_l}u_{x_lx_j}+H^i_{p_kx_j})(H^i_{p_jp_m}u_{x_mx_k}+H^i_{p_jx_k})\\&
\ge (\theta^i)^{p+1}[\gamma^2|D^2u^i|^2-C(1+|Du^i|)|D^2u^i|-C(1+|Du^i|^2)]\\&
\ge  (\theta^i)^{p+1}\tilde{\gamma}^2|D^2u^i|^2 -C(\theta^i)^{p+1}(1+|Du^i|^2)
\end{split}
\eeqs for some $\tilde{\gamma}>0$. Clearly, $$I_2=p(\theta^i)^{p-1}|D_pH^i\cdot D\theta^i|^2.$$ Next, we estimate $I_3$ from below. 
From equation \eqref{obstacleepmfg1sys}, we gather that
\beq\label{gradthetasunstitution}H^i_{p_j}u^i_{x_jx_l}=g'(\theta^i)\theta^i_{x_l}-H^i_{x_l}-u^i_{x_l}-\sum_{j\neq i}\beta_\ep'(u^i-u^j-\psi^{ij})(u^i-u^j-\psi^{ij})_{x_l}.\eeq
The estimate \eqref{g'prop} in Assumption \ref{A4}  and the lower bound \eqref{thelowboundepasslem} on $\theta^i$
imply the existence of a positive constant, $C_0$ (depending on $\ep$), such that
\beq\label{g'thetathetalower}g'(\theta^i)\theta^i\ge C_0>0.\eeq 
Then, using \eqref{Hconvexity}, \eqref{growthH}, \eqref{beta'bounded}, \eqref{gradthetasunstitution},  \eqref{g'thetathetalower}, and the Cauchy-Schwarz inequality, we get
\beqs\begin{split}I_3&=(p+1)(\theta^i)^p\theta^i_{x_k}H^i_{p_j}(H^i_{p_kp_l}u^i_{x_lx_j}+H^i_{p_kx_j})\\&
=(p+1)g'(\theta^i)(\theta^i)^{p}H^i_{p_kp_l}\theta^i_{x_k}\theta^i_{x_l}+(p+1)(\theta^i)^p\theta^i_{x_k}(H^i_{p_j}H^i_{p_kx_j}-H^i_{p_kp_l}H^i_{x_l})\\&
-(p+1)(\theta^i)^pH^i_{p_kp_l}\theta^i_{x_k}\left(u^i_{x_l}+\sum_{j\neq i}\beta_\ep'(u^i-u^j-\psi^{ij})(u^i-u^j-\psi^{ij})_{x_l}\right)\\&
\ge (p+1)\gamma C_0(\theta^i)^{p-1}|D\theta^i|^2-C(p+1)(\theta^i)^p|D\theta^i|(1+|Du^i|^2)\\&
-C(p+1)(\theta^i)^p|D\theta^i|(1+|Du^i|)-\sum_{j\neq i}C(p+1)(\theta^i)^p|D\theta^i|(1+|Du^j|)\\&
\ge C(p+1)(\theta^i)^{p-1}|D\theta^i|^2-C(p+1)(\theta^i)^{p+1}\left(1+|Du^i|^4+\sum_{j\neq i}|Du^j|^2\right).
\end{split}
\eeqs

Next, we bound the left-hand side of \eqref{lemlipine1} from above. Using \eqref{growthH},  \eqref{beta'bounded}, and the Cauchy-Schwarz inequality, we obtain
\beqs\begin{split} &(\theta^i+\sum_{j\neq i}\beta_\ep'(u^i-u^j-\psi^{ij})\theta^i)\di ((\theta^i)^p D_p H^i)\\&
= [\theta^i+\sum_{j\neq i}\beta_\ep'(u^i-u^j-\psi^{ij})\theta^i][p(\theta^i)^{p-1}D_pH^i \cdot D\theta^i+(\theta^i)^pH^i_{p_kp_j}u^i_{x_jx_k}+(\theta^i)^pH^i_{p_kx_k}]\\&
\leq C\theta^i[p(\theta^i)^{p-1}|D_pH^i \cdot D\theta^i|+(\theta^i)^p |H^i_{p_kp_j}u^i_{x_jx_k}|+(\theta^i)^p|H^i_{p_kx_k}|]\\&
\le \frac{p}{2} (\theta^i)^{p-1}|D_pH^i \cdot D\theta^i|^2+\frac{\tilde{\gamma}^2}{2}(\theta^i)^{p+1}|D^2u^i|^2+Cp(\theta^i)^{p+1}(1+|Du^i|).
\end{split}
\eeqs
Similarly, 
\beqs\begin{split} &-\sum_{j\neq i}\beta_\ep'(u^j-u^i-\psi^{ji})\theta^j \di ((\theta^i)^p D_pH^i)\\&
\leq \sum_{j\neq i}C\theta^j\left|p(\theta^i)^{p-1}|D_pH^i \cdot D\theta^i|+(\theta^i)^p|H^i_{p_kp_j}u^i_{x_jx_k}|+(\theta^i)^p|H^i_{p_kx_k}|\right|\\&
\leq\sum_{j\neq i}
C\theta^j 
\left[
p (\theta^i)^\frac{p-1}{2} (\theta^i)^\frac{p-1}{2}\left|D_pH^i \cdot D\theta^i\right|+(\theta^i)^\frac{p-1}{2} (\theta^i)^\frac{p+1}{2}\left|H^i_{p_kp_j}u^i_{x_jx_k}\right|
+(\theta^i)^\frac{p-1}{2} (\theta^i)^\frac{p+1}{2}\left|H^i_{p_kx_k}\right|
\right]\\&
\leq \sum_{j\neq i}Cp(\theta^j)^2(\theta^i)^{p-1}+\frac{p}{2} (\theta^i)^{p-1}|D_pH^i \cdot D\theta^i|^2+\frac{\tilde{\gamma}^2}{2}(\theta^i)^{p+1}|D^2u^i|^2+Cp(\theta^i)^{p+1}(1+|Du^i|^2).
\end{split}
\eeqs
From the preceding estimates, we conclude that 
\beqs\begin{split}
&C (p+1)\intor(\theta^i)^{p-1}|D\theta^i|^2dx-C(p+1)\intor(\theta^i)^{p+1}(1+|Du^i|^4+\sum_{j\neq i}|Du^j|^2)dx\\&+p\intor(\theta^i)^{p-1}|D_pH^i\cdot D\theta^i|^2 dx\\&
+\tilde{\gamma}^2\intor (\theta^i)^{p+1}|D^2u^i|^2 -C\intor (\theta^i)^{p+1}(1+|Du^i|^2)dx\\&
\le \intor I_1+I_2+I_3dx\\&
=\intor [\theta^i+\sum_{j\neq i}\beta_\ep'(u^i-u^j-\psi^{ij})\theta^i-\sum_{j\neq i}\beta_\ep'(u^j-u^i-\psi^{ji})\theta^j]\di ((\theta^i)^p D_pH^i)dx\\&
\leq  p\intor (\theta^i)^{p-1}|D_pH^i\cdot D\theta^i|^2dx+\tilde{\gamma}^2\intor(\theta^i)^{p+1}|D^2u^i|^2dx\\&+Cp\intor(\theta^i)^{p+1}(1+|Du^i|^2)dx+\sum_{j\neq i}Cp\intor (\theta^j)^2(\theta^i)^{p-1}.
\end{split}
\eeqs 
Consequently, 
\beqs\begin{split} 
\intor(\theta^i)^{p-1}|D\theta^i|^2dx\leq \intor(\theta^i)^{p+1}(1+|Du^i|^4+\sum_{j\neq i}|Du^j|^2)dx+ \sum_{j\neq i}C\intor (\theta^j)^2(\theta^i)^{p-1}dx.
\end{split}
\eeqs 
Applying Young's inequality, we gather
\beqs (\theta^j)^2(\theta^i)^{p-1}\leq \frac{2}{p+1}(\theta^j)^{p+1}+\frac{p-1}{p+1}(\theta^i)^{p+1}.
\eeqs
Therefore, 
\beq\label{thetagradpropthmclaim1}\begin{split} 
\intor(\theta^i)^{p-1}|D\theta^i|^2dx\leq \intor(\theta^i)^{p+1}(1+|Du^i|^4+\sum_{j\neq i}|Du^j|^2)dx+\sum_{j\neq i}C\intor (\theta^j)^{p+1}dx.
\end{split}\eeq
If $N=1$,    \eqref{thetabounded} is a consequence of  Morrey's Theorem.
If  $N=2$, then $\theta^i\in L^p$ for all $p$. Moreover, in this case,  
the argument that follows can be modified by replacing the Sobolev exponent, $2^*,$ by any arbitrarily large number, $M$. 
Therefore, we assume that  $N>2$. Accordingly, 
by \eqref{w12alpthet},  we have $\theta^i\in L^{\frac{2^* (1+\alpha)}{2}}$. 
In addition, Sobolev's inequality provides the bound
\beq\label{sobolinthmlip}\begin{split} \left(\intor (\theta^i)^{\frac{p+1}{2}2^*}dx\right)^\frac{2}{2^*}&
\leq C\intor (\theta^i)^{p+1}dx+C\intor|D((\theta^i)^{\frac{p+1}{2}})|^2dx\\&
= C\intor (\theta^i)^{p+1}dx+C(p+1)^2\intor(\theta^i)^{p-1}|D\theta^i|^2dx.
\end{split}\eeq

Let $\beta:=\sqrt{\frac{2^*}{2}}=\sqrt{\frac{N}{N-2}}>1$. Consider first the {\bf P} case. Then, Assumption \ref{A7} implies that $$2\al\leq (\alpha+1)\beta^2\frac{\beta-1}{\beta}.$$
The previous inequality together with \eqref{Hquadratic} implies that 
\beqs|Du^i|^4\leq C(g(\theta^i))^2+C\leq C(\theta^i)^{2\al}+C\leq C(1+(\theta^i)^{(\alpha+1)\beta^2\frac{\beta-1}{\beta}}).\eeqs 
The same inequality holds in the logarithmic case with  $\al=0$:
 \beqs|Du^i|^4\leq C(g(\theta^i))^2+C\leq C(\log(\theta^i))^2+C\leq C(1+(\theta^i)^{\beta^2\frac{\beta-1}{\beta}}).\eeqs
 Therefore, from H{\"o}lder's inequality, we get
 \beqs\begin{split}
        &\intor (\theta^i)^{p+1}(1+|Du^i|^4)dx\leq C\intor (\theta^i)^{p+1}(1+(\theta^i)^{(\alpha+1)\beta^2\frac{\beta-1}{\beta}})dx\\&
 \leq C \intor (\theta^i)^{p+1}dx+C\left(\intor (\theta^i)^{(p+1)\beta}dx\right)^\frac{1}{\beta}\left(\intor (\theta^i)^{(\alpha+1)\beta^2} dx\right)^\frac{\beta-1}{\beta}\\&
 \leq C \intor (\theta^i)^{p+1}dx+C\left(\intor (\theta^i)^{(p+1)\beta}dx\right)^\frac{1}{\beta}\\&
 \leq C\left(\intor (\theta^i)^{(p+1)\beta}dx\right)^\frac{1}{\beta}.
\end{split} \eeqs
Similarly,
 \beqs\begin{split}&\intor (\theta^i)^{p+1}(1+|Du^j|^2)dx\leq C\intor (\theta^i)^{p+1}(1+(\theta^j)^{(\alpha+1)\beta^2\frac{\beta-1}{\beta}})dx\\&
 \leq C \intor (\theta^i)^{p+1}dx+C\left(\intor (\theta^i)^{(p+1)\beta}dx\right)^\frac{1}{\beta}\left(\intor (\theta^j)^{(\alpha+1)\beta^2} dx\right)^\frac{\beta-1}{\beta}\\&
 \leq C \intor (\theta^i)^{p+1}dx+C\left(\intor (\theta^i)^{(p+1)\beta}dx\right)^\frac{1}{\beta}\\&
 \leq C\left(\intor (\theta^i)^{(p+1)\beta}dx\right)^\frac{1}{\beta}.
\end{split} \eeqs

The last two inequalities, combined with \eqref{thetagradpropthmclaim1} and \eqref{sobolinthmlip} give the bound 
\beqs\left(\intor(\theta^i)^{(p+1)\beta^2}dx\right)^\frac{1}{\beta^2}\leq Cp^2 \left(\intor(\theta^i)^{(p+1)\beta}dx\right)^\frac{1}{\beta}+Cp^2\sum_{j\neq i} \left(\intor(\theta^j)^{(p+1)\beta}dx\right)^\frac{1}{\beta}
\eeqs
for $i=1,\ldots,d$.
Summing on $i$, we finally obtain
\beq\label{thetaboundedbeta}\sum_{i=1}^d\left(\intor(\theta^i)^{(p+1)\beta^2}dx\right)^\frac{1}{\beta^2}\leq Cp^2 \sum_{i=1}^d\left(\intor(\theta^i)^{(p+1)\beta}dx\right)^\frac{1}{\beta}.\eeq
Arguing as in \cite{E1}, we get 
$$\sum_{i=1}^d\|\theta^i\|_{L^\infty(\Tt^N)}\leq C$$ and, hence,
  \eqref{thetabounded}. 
%
\end{proof}
 
\begin{cor}
        Suppose that Assumptions \ref{A1}-\ref{A5}  and \ref{B1} hold, and either
\begin{itemize}
        \item[-]  Assumption \ref{A5} {\bf L} or
        \item[-]  Assumptions \ref{A5} {\bf P-$\frac 2 N$}, \ref{A6} and \ref{A7} 
\end{itemize}   
        hold. Let $(u,\theta)$ be a $C^\infty$ solution of \eqref{obstacleepmfg1sys}-\eqref{obstacleepmfg2sys}.
        Then, for any $k\in\N,$  there exists $C_{\ep,k}>0$ that does not depend on the particular solution, such that
        \beq\label{cinftyestimate}\|u^i\|_{W^{k,\infty}(\Tt^N)}+\|\theta^i\|_{W^{k,\infty}(\Tt^N)}\leq C_{\ep,k},\eeq
        for any $i=1,\ldots,d$.
\end{cor}       
\begin{proof}
If Assumption \ref{A5} {\bf L} holds or if
 Assumptions \ref{A5} {\bf P-$\frac 2 N$} and \ref{A6} hold 
by, respectively,  Propositions \ref{prolog}, 
and \ref{MEst}, 
then there exists $\theta_0>0$, 
 such that, for any $i=1,\ldots,d$ and any $x\in \Tt^N$, 
 \beq\label{theta>theta0anyorderlem}\theta^i(x)\geq \theta_0>0.
 \eeq
Thus, we use
$$\theta^i=g^{-1}\left(H^i(Du^i,x)+u^i+\sum_{j\ne i}\beta_\ep(u^i-u^j-\psi^{ij})\right)$$
in \eqref{obstacleepmfg2sys} to get
\beqs\begin{split}&-\theta^i H^i_{p_kp_j}u^i_{x_jx_k}-\theta^i H^i_{p_kx_k}-\frac{1}{g'(\theta^i)}\big[H^i_{p_k}H^i_{p_j}u^i_{x_jx_k}+H^i_{x_k}H^i_{p_k}+H^i_{p_k}u^i_{x_k}\\&+
\sum_{j\neq i}\beta_\ep'(u^i-u^j-\psi^{ij})(u^i-u^j-\psi^{ij})_{x_k}H^i_{p_k}\big]+\theta^i+\sum_{j\neq i}\beta_\ep'(u^i-u^j-\psi^{ij})\theta^i\\&
=\sum_{j\neq i}\beta_\ep'(u^j-u^i-\psi^{ji})\theta^j+1.
\end{split}\eeqs
From estimates \eqref{thetabounded},  \eqref{lipepest}, and  \eqref{theta>theta0anyorderlem}, 
we have that the previous equation is a uniformly elliptic equation for each $i$.
Therefore, from the elliptic regularity theory, we infer that 
$$\|u^i\|_{W^{2,p}(\Tt^N)}+\|\theta^i\|_{W^{1,p}(\Tt^N)}\leq C_{\ep}$$ for any $1<p<\infty$. Repeated differentiation and a bootstrapping argument give 
\eqref{cinftyestimate}. 
\end{proof}

\section{Proof of Theorem \ref{main2}}
\label{pmain2}

In this  section, 
we show the existence and uniqueness of a  classical solution of \eqref{obstacleepmfg1sys}-\eqref{obstacleepmfg2sys}.
In the proof of existence, we use the continuation method. In  the proof of uniqueness, we rely on a monotonicity argument. 
Here, we work under Assumptions \ref{A1}-\ref{A4} and either
\ref{A5} {\bf L} or \ref{A5} {\bf P-}$\frac 2 N$ together with Assumptions
\ref{A6} and \ref{A7}. 

\subsection{Existence}
 To prove  the existence of a  classical solution of \eqref{obstacleepmfg1sys}-\eqref{obstacleepmfg2sys}  using the continuation method, we define\beqs H^i_\lambda(p,x):=\lambda H^i(p,x)+(1-\lambda)\frac{|p|^2}{2}\eeqs 
for $0\leq \lambda\leq 1$, $(p,x)\in\R^N\times \Tt^N$, and  $i=1,\ldots,d$. We introduce the mean-field game
\beq\label{obstacleepmfg1syslam} H^i_\lambda(Du_\lambda^i,x)+u_\lambda^i+\sum_{j\ne i}\beta_\ep(u_\lambda^i-u_\lambda^j-\psi^{ij}) =g(\theta_\lambda^i),\eeq
\beq\label{obstacleepmfg2syslam} -\di (D_p H_\lambda^i(Du_\lambda^i,x)\theta_\lambda^i)+\theta_\lambda^i
+\sum_{j\ne i}\beta'_\ep(u_\lambda^i-u_\lambda^j-\psi^{ij}) \theta^i_\lambda -\beta'_\ep(u_\lambda^j-u_\lambda^i-\psi^{ji}) \theta^j_\lambda =1 
\eeq
for $x\in\Tt^N$ and $i=1,\ldots, d$.

Next, for $k\in\N$, we set $E^k:=(H^{k}(\Tt^N))^{d}$, $E^0:=  (L^2(\Tt^N))^{d}$. If $k>\frac{N}{2}$, $E^k$ is an algebra. Moreover,    $E^k\subset \left(C^\gamma(\Tt^N)\right)^d$ for any 
 $0<\gamma<2-\frac{N}{k}$. Given $\theta_0>0$ and  $k>\frac{N}{2}$, we set
 $$E^k_{\theta_0}:=\{\theta \in E^k\,|\, \theta^i\geq\theta_0,\,i=1,\ldots, d\}.$$
Finally, for $k>\frac{N}{2}$, we define  $F:[0,1]\times E^{k+2}\times E_{\theta_0}^{k+1}\to E^k\times E^{k+1}$ by 
\beqs\begin{split} &F(\lambda, u, \theta):=\\&\left(\begin{array}{ccc}
\di (D_p H_\lambda^i(Du_\lambda^i,x)\theta_\lambda^i)-\theta_\lambda^i
-\sum_{j\ne i}(\beta'_{\ep,\lambda}(u_\lambda^i-u_\lambda^j-\psi^{ij}) \theta^i_\lambda -\beta'_{\ep,\lambda}(u_\lambda^j-u_\lambda^i-\psi^{ji}) \theta^j_\lambda) +1\\
H^i_\lambda(Du_\lambda^i,x)+u_\lambda^i+\sum_{j\ne i}\beta_\ep(u_\lambda^i-u_\lambda^j-\psi^{ij}) -g(\theta_\lambda^i)
\end{array}\right).
\end{split}
\eeqs
Then, \eqref{obstacleepmfg1syslam}-\eqref{obstacleepmfg2syslam} is equivalent to
$$F(\lambda, u_\lambda, \theta_\lambda)=0.$$
Let
\beqs\Lambda:=\{\lambda\in[0,1]\,|\,\text{ \eqref{obstacleepmfg1syslam}-\eqref{obstacleepmfg2syslam} has a classical solution } (u_\lambda,\theta_\lambda)\}.\eeqs 
Next, we show that $$\Lambda=[0,1].$$ 
We divide the proof of this identity into 
the three following claims. 
\bigskip

{\em Claim 1:}  $0\in\Lambda$. Indeed, for $\lambda=0,$ we have the explicit solution:
$$u_0^i=g(1),\quad\theta_0^i=1,\quad i=1,\ldots, d.$$
\bigskip

{\em Claim 2:  $\Lambda$ is closed.} To prove this claim, we show that, for any sequence, $(\lambda_k)_k\subset\Lambda,$ such that $\lambda_k\to\lambda_0$ 
as $k\to+\infty$, we have that $\lambda_0\in\Lambda$. Accordingly, let $(u_{\lambda_k},\theta_{\lambda_k})$ be a classical solution of  
\eqref{obstacleepmfg1syslam}-\eqref{obstacleepmfg2syslam} for $\lambda=\lambda_k$. Recall that  $H_\lambda$ satisfies Assumptions  \ref{A1}-\ref{A4} and Assumption \ref{A6} uniformly in $0\leq \lambda\leq 1$.
Therefore, by \eqref{cinftyestimate}, we can bound the derivatives of any order of $(u_{\lambda_k},\theta_{\lambda_k})$ by a constant that is independent of $k$. Consequently, we can extract a subsequence of smooth solutions converging to a limit function $(u,\theta)$ that solves  \eqref{obstacleepmfg1syslam}-\eqref{obstacleepmfg2syslam} with $\lambda=\lambda_0$. Thus, $\lambda_0\in\Lambda$. 
\bigskip

{\em Claim 3: $\Lambda$ is open.} To prove this last claim, we need to check that, for any $\lambda_0\in\Lambda,$ there exists a neighborhood of $\lambda_0$ contained in $\Lambda$.
To do so, we use the implicit function theorem.  To simplify the notation, for $h=\beta,\beta',\beta''$, we set
 $$h_{\ep,\lambda_0}(i,j):=h_{\ep}(u_{\lambda_0}^i-u_{\lambda_0}^j-\psi^{ij}).$$
For $\lambda_0\in\Lambda$, we consider the Fr\'{e}chet derivative, $\mathcal{L}_{\lambda_0}:E^{k+2}\times E^{k+1}\to E^k\times E^{k+1}$, of $(u,\theta)\to F(\lambda_0,u,\theta)$ at 
$(u_{\lambda_0},\theta_{\lambda_0})$. We have
\beq
\label{ELLF}
\begin{split}
        &\mathcal{L}_{\lambda_0}(v,f)\\&
=\left(\begin{array}{ccc}
(H^i_{\lambda_0,p_kp_j}(Du^i_{\lambda_0},x)v^i_{x_j}\theta_{\lambda_0}^i+H^i_{\lambda_0,p_k}(Du^i_{\lambda_0},x)f^i)_{x_k}-f^i
\\-\sum_{j\neq i}[(\beta''_{\ep,\lambda_0}(i,j)\theta_{\lambda_0}^i+\beta''_{\ep,\lambda_0}(j,i)\theta_{\lambda_0}^j)(v^i-v^j)+\beta'_{\ep,\lambda_0}(i,j)f^i-\beta'_{\ep,\lambda_0}(j,i)f^j]\\
\,\\
D_pH^i_{\lambda_0}(Du^i_{\lambda_0},x)\cdot Dv^i+v^i+\sum_{j\neq i}(\beta'_{\ep,\lambda_0}(i,j)(v^i-v^j))-g'(\theta^i_{\lambda_0})f^i\
\end{array}
\right).
\end{split}
\eeq

Because of the a priori bounds for smooth solutions \eqref{cinftyestimate}
and either estimate \eqref{lowerboundthelog}, in the {\bf L} case, or \eqref{theta>1power}, in the {\bf P-}$\frac 2 N$ case, the operator, \(\mathcal{L}_{\lambda_0}\) , is  well defined in 
$E^{k+2}\times E^{k+1}$ for any $k\geq 0$. Next, we
prove that $\mathcal{L}_{\lambda_0}$ is an isomorphism from $E^{k+2}\times E^{k+1}$ to $E^{k}\times E^{k+1}$ for any $k\geq 0$. 

Let $w=(v,f)\in E^1\times E^0$. Define the bilinear form, $B_{\lambda_0}[w_1,w_2]:E^1\times E^0\to \R,$ by
\beq
\label{bdef}
 B_{\lambda_0}[w_1,w_2]:=\sum_{i=1}^d B^i_{\lambda_0}[w_1,w_2],
 \eeq
  where
\beqs\begin{split}& B^i_{\lambda_0}[w_1,w_2]:=\\&
\intor-H^i_{\lambda_0,p_kp_j}(Du^i_{\lambda_0},x)v^i_{1,x_j}v_{2,x_k}^i\theta_{\lambda_0}^i-D_pH^i_{\lambda_0}(Du^i_{\lambda_0},x)\cdot Dv_{2}^i f_1^i-f^i_1v^i_2dx\\&
-\intor\Big[\sum_{j\neq i}(\beta''_{\ep,\lambda_0}(i,j)\theta_{\lambda_0}^i+\beta''_{\ep,\lambda_0}(j,i)\theta_{\lambda_0}^j)(v_1^i-v_1^j)+\beta'_{\ep,\lambda_0}(i,j)f_1^i-\beta'_{\ep,\lambda_0}(j,i)f_1^j\Big]v^i_2 dx\\&
+\intor [D_pH^i_{\lambda_0}(Du^i_{\lambda_0},x)\cdot Dv^i_1+v_1^i+\sum_{j\neq i}(\beta'_{\ep,\lambda_0}(i,j)(v_1^i-v_1^j))-g'(\theta^i_{\lambda_0})f_1^i]f^i_2dx.
 \end{split}\eeqs
If $w_1\in E^{k+2}\times  E^{k+1}$ with $k\geq 0$, then 
$$ B_{\lambda_0}[w_1,w_2]=\intor \mathcal{L}_{\lambda_0}(w_1)\cdot w_2 dx.$$
The following lemma  is a straightforward consequence  of estimate  \eqref{cinftyestimate} combined with  either  \eqref{lowerboundthelog} or \eqref{theta>1power}.
\begin{lem} Let $B$ be the bilinear form given by \eqref{bdef}. Then, there exists $C>0$ such that 
\beqs |B_{\lambda_0}[w_1,w_2]|\leq C\|w_1\|_{E^1\times E^0}\|w_2\|_{E^1\times E^0}\eeqs
for any $w_1,w_2\in E^1\times E^0$.
\end{lem}

Thus, in view of Riesz's representation theorem for Hilbert spaces, there exists a continuous linear mapping, $A:E^1\times E^0\to E^1\times E^0,$ such that 
\beq\label{BA}B_{\lambda_0}[w_1,w_2]=(Aw_1,w_2)_{E^1\times E^0}.\eeq
\begin{lem}\label{Ainjectivelem}  The operator, $A,$ defined in \eqref{BA} is injective.
\end{lem}
\begin{proof}
Let $w=(v,f)$. Then, we have
\beqs\begin{split}
&B^i_{\lambda_0}[w,w]=
\intor-H^i_{\lambda_0,p_kp_j}(Du^i_{\lambda_0},x)v^i_{x_j}v_{x_k}^i\theta_{\lambda_0}^i\\&
-\intor\sum_{j\neq i}(\beta''_{\ep,\lambda_0}(i,j)\theta_{\lambda_0}^i+\beta''_{\ep,\lambda_0}(j,i)\theta_{\lambda_0}^j)(v^i-v^j)v^i+g'(\theta^i_{\lambda_0})(f^i)^2dx.
\end{split}\eeqs
Summing the preceding expression on $i$, using the identity 
\begin{align*}
&\sum_i\sum_{j\neq i}(\beta''_{\ep,\lambda_0}(i,j)\theta_{\lambda_0}^i+\beta''_{\ep,\lambda_0}(j,i)\theta_{\lambda_0}^j)(v^i-v^j)v^i\\
&=\sum_i\sum_{j\neq i}\beta''_{\ep,\lambda_0}(i,j)\theta_{\lambda_0}^i (v^i-v^j)^2,
\end{align*}
 the convexity property \eqref{Hconvexity} from Assumption \ref{A4} and either
\eqref{lowerboundthelog}, in the {\bf L} case, or \eqref{theta>1power}, in the {\bf P-}$\frac 2 N$ case,
we get 
\beq\begin{split}
\label{u1}
B_{\lambda_0}[w,w]&=\sum_{i=1}^d B^i_{\lambda_0}[w,w]\\&
\leq -\sum_{i=1}^d\intor H^i_{\lambda_0,p_kp_j}(Du^i_{\lambda_0},x)v^i_{x_j}v_{x_k}^i\theta_{\lambda_0}^i +g'(\theta^i_{\lambda_0})(f^i)^2dx\\&
-\sum_i\sum_{j\neq i} \intor\beta''_{\ep,\lambda_0}(i,j)\theta_{\lambda_0}^i (v^i-v^j)^2
\\&
\leq -C_{\lambda_0} \intor \|Dv\|^2+\|f\|^2 dx\\&
-\theta_0\sum_i\sum_{j\neq i}\int_{\Tt^N}\beta''_{\ep,\lambda_0}(i,j)(v^i-v^j)^2\, dx.
\end{split}
\eeq
According to the previous inequality, if $Aw=0$, we have $w=(\mu,0)$ for some $\mu\in\R^d$. Next, by computing
$$0=(Aw,(0,\mu))=B[(\mu,0),(0,\mu)]=\sum_{i=1}^d\intor v_1^if_2^idx=|\mu|^2,$$
we conclude that $\mu=0$.
\end{proof}

\begin{lem}\label{Asurjlem}  The operator, $A,$ given by \eqref{BA} is surjective. 
\end{lem}
\begin{proof}
First, we prove that the range of $A$ is closed in $E^1\times E^0$.
For that, take a Cauchy sequence, $(z_n)_n$, in the range of $A$; that is, $z_n=Aw_n$ for some  sequence
$(w_n)_n$  in $E^1\times E^0$. We claim that 
$(w_n)_n$ is a Cauchy sequence. Let $w_n=(v_n,f_n)$.  Then, according to \eqref{u1}, we have
\beqs\begin{split}(z_n-z_m, w_n-w_m)_{E^1\times E^0}&=(A(w_n-w_m),w_n-w_m)_{E^1\times E^0}\\&\qquad=B[w_n-w_m,w_n-w_m]\\&
\leq -C(\|D(v_n-v_m)\|_{E^0}^2+\|f_n-f_m\|_{E^0}^2)\\&
-\sum_i\sum_{j\neq i}\int_{\Tt^N}\beta''_{\ep,\lambda_0}(i,j)((v^i_n-v^i_m)-(v^j_n-v^j_m))^2\,
dx.
\end{split}
\eeqs
Moreover, we have
\beqs\begin{split}|(z_n-z_m, &w_n-w_m)_{E^1\times E^0}|\leq \|z_n-z_m\|_{E^1\times E^0}\|w_n-w_m\|_{E^1\times E^0}\\&
= \|z_n-z_m\|_{E^1\times E^0}(\|v_n-v_m\|_{ E^0}+\|Dv_n-Dv_m\|_{ E^0}+\|f_n-f_m\|_{ E^0})\\&
\leq \delta( \|Dv_n-Dv_m\|_{ E^0}^2+\|f_n-f_m\|_{ E^0}^2)+C_\delta \|z_n-z_m\|_{E^1\times E^0}^2\\&+\|z_n-z_m\|_{E^1\times E^0}\|v_n-v_m\|_{ E^0}.
\end{split}\eeqs
Let $ \mu$ be a positive constant to be chosen later. By selecting a suitably small $\delta$ and combining the inequalities above, we get
\beq\label{lemmasurjectA1}\begin{split}
& \|Dv_n-Dv_m\|_{ E^0}^2+\|f_n-f_m\|_{ E^0}^2\\&
 +\sum_i\sum_{j\neq i}\int_{\Tt^N\cap \{u^i-u^j-\psi^{ij}>0\}}\beta''_{\ep,\lambda_0}(i,j)((v^i_n-v^i_m)-(v^j_n-v^j_m))^2\,
dx\\ &
 \qquad \leq C \|z_n-z_m\|_{E^1\times E^0}^2+\|z_n-z_m\|_{E^1\times E^0}\|v_n-v_m\|_{ E^0}\\
 &\qquad \leq C_\mu \|z_n-z_m\|_{E^1\times E^0}^2+\mu \|v_n-v_m\|_{ E^0}^2. 
 \end{split}\eeq
 Next, we have 
\beqs\begin{split}
B[w_n-w_m&,(0,v_n-v_m)]=
\\&
\intor  \sum_{i=1}^d 
\left[
D_pH^i_{\lambda_0}(Du^i_{\lambda_0},x)\cdot \left(D(v^i_n-v^i_m)\right)(v_n^i-v_m^i)+(v^i_n-v^i_m)^2
\right]dx\\&
+\intor  \sum_{i=1}^d\sum_{i\neq j}\beta'_{\ep,\lambda_0}(i,j)
[
(v_n^i-v_m^i)
-(v_n^j-v_m^j)](v_n^i-v_m^i)\, dx\\&
-\intor  \sum_{i=1}^dg'(\theta^i_{\lambda_0})(f_n^i-f_n^i)(v_n^i-v_m^i)\, dx.
\end{split}\eeqs
Set
\[
M=\left[\sum_{i=1}^d
\int_{\Tt^N}\left(\sum_{i\neq j}\beta'_{\ep,\lambda_0}(i,j)
[
(v_n^i-v_m^i)
-(v_n^j-v_m^j)]\right)^2\right]^{1/2}.
\]
Then, using   \eqref{thetabounded},  \eqref{lipepest}, H\"{o}lder's inequality,         and Young's inequality, we get
\beq\label{lstineqlemmsurj}\begin{split}&B[w_n-w_m,(0,v_n-v_m)]\\&\geq \|v_n-v_m\|_{ E^0}^2\\&-C  (\|Dv_n-Dv_m\|_{ E^0}^2+ \|v_n-v_m\|_{ E^0} \|f_n-f_m\|_{ E^0}+M\|v_n-v_m\|_{
E^0})\\&
\geq  C\|v_n-v_m\|_{ E^0}^2-C  (\|Dv_n-Dv_m\|_{ E^0}^2+  \|f_n-f_m\|_{ E^0}^2+M^2).
\end{split}\eeq
On the other hand, \eqref{BA}   implies
that \beqs B[w_n-w_m,(0,v_n-v_m)]\leq C\|z_n-z_m\|_{E^1\times E^0}\|v_n-v_m\|_{E^0}.\eeqs
Therefore, 
\begin{align*}
&C\|v_n-v_m\|_{ E^0}^2-C  (\|Dv_n-Dv_m\|_{ E^0}^2+  \|f_n-f_m\|_{ E^0}^2+M^2)
\\&
\leq C\|z_n-z_m\|_{E^1\times E^0}\|v_n-v_m\|_{E^0}\\
&\leq C\|z_n-z_m\|_{E^1\times E^0}^2 +\mu_2 \|v_n-v_m\|_{E^0}^2 
\end{align*}
for some $\mu_2$ to be selected later. Next, taking into account that 
\[
M^2\leq C
\sum_i\sum_{j\neq i}\int_{\Tt^N}\beta''_{\ep,\lambda_0}(i,j)[(v^i_n-v^i_m)-(v^j_n-v^j_m)]^2\,
dx
\]
and using \eqref{lemmasurjectA1} and \eqref{lstineqlemmsurj},
 we obtain
\begin{align*}
\|v_n-v_m\|_{ E^0}^2\leq C\|z_n-z_m\|_{E^1\times E^0.}^2  
\end{align*}
Consequently, $v_n$\ is a Cauchy sequence in $L^2$. Finally, from 
\eqref{lemmasurjectA1} we infer that $w_n$ is a Cauchy sequence in $E^1\times E^0$ because of
the bound\beqs  \|w_n-w_m\|_{E^1\times E^0}^2\leq C  \|z_n-z_m\|_{E^1\times E^0}^2.\eeqs The last inequality and the continuity of $A$ imply that $R(A)$ is closed. 

Finally, we prove that $R(A)=E^1\times E^0$. Suppose that $R(A)\neq E^1\times E^0$. Since $R(A)$ is closed, there exists $z\in R(A)^\perp$ with $z\neq 0$, such that 
$B_{\lambda_0}[z,z]=0$. The argument in  the proof of Lemma \ref{Ainjectivelem} implies that  $z=0$, which is a contradiction.
\end{proof}

\begin{lem}\label{Lisomorphlem} The operator, $\mathcal{L}_{\lambda_0}:E^{k+2}\times E^{k+1}\to E^k\times E^{k+1}$,
given by \eqref{ELLF}   
         is an isomorphism for all $k\in \N$.
\end{lem}
\begin{proof} 
By Lemma \ref{Ainjectivelem}, $\mathcal{L}_{\lambda_0}$ is injective. Therefore, it suffices to prove that $\mathcal{L}_{\lambda_0}$ is surjective. To do so, we fix $w_0\in E^0\times E^{1}$ with 
$w_0=(v_0,f_0)$. We claim that there exists a solution, $w_1\in E^{2}\times E^1,$ to $\mathcal{L}_{\lambda_0}w_1=w_0$.

Consider the bounded linear functional, $w\to(w_0,w)_{(E^0)^2}$, in $E^1\times E^0$. According to the Riesz representation theorem, there exists $\widetilde{w}\in E^1\times E^0,$ such that 
$(w_0,w)_{(E^0)^2}=(\widetilde{w},w)_{E^1\times E^0}$ for any $w\in E^1\times E^0$. In light of Lemmas \ref{Ainjectivelem} and \ref{Asurjlem}, 
the operator $A$ is invertible in $E^1\times E^0$. We define $w_1:=A^{-1}\widetilde{w}$. Then, for any $w\in E^1\times E^0,$
$$(Aw_1,w)_{E^1\times E^0}=(w_0,w)_{E^0\times E^0};$$ that is, $w_1=(v,f)$ is a weak solution of 
\beq\label{lw1-w0}\begin{cases}
(H^i_{\lambda_0,p_kp_j}(Du^i_{\lambda_0},x)v^i_{x_j}\theta_{\lambda_0}^i+H^i_{\lambda_0,p_k}(Du^i_{\lambda_0},x)f^i)_{x_k}-f^i
\\-\sum_{j\neq i}[(\beta''_{\ep,\lambda_0}(i,j)\theta_{\lambda_0}^i+\beta''_{\ep,\lambda_0}(j,i)\theta_{\lambda_0}^j)(v^i-v^j)+\beta'_{\ep,\lambda_0}(i,j)f^i-\beta'_{\ep,\lambda_0}(j,i)f^j]=v_0^i\\
\,\\
D_pH^i_{\lambda_0}(Du^i_{\lambda_0},x)\cdot Dv^i+v^i+\sum_{j\neq i}(\beta'_{\ep,\lambda_0}(i,j)(v^i-v^j))-g'(\theta^i_{\lambda_0})f^i=f_0^i
\end{cases}
\eeq
for $x\in\Tt^N$ and $i=1, \hdots, d$. 
 From the second equation in \eqref{lw1-w0}, we obtain
\beq\label{fiexplic}f^i=(g'(\theta^i_{\lambda_0}))^{-1}(D_pH^i_{\lambda_0}(Du^i_{\lambda_0},x)\cdot Dv^i+v^i+\sum_{j\neq i}(\beta'_{\ep,\lambda_0}(i,j)(v^i-v^j))-f_0^i).\eeq
Using \eqref{fiexplic} in the first equation of \eqref{lw1-w0},
we see that $v^i$ is a weak solution to

\beqs\begin{split}
&[H^i_{\lambda_0,p_kp_j}(Du^i_{\lambda_0},x)v^i_{x_j}\theta_{\lambda_0}^i+H^i_{\lambda_0,p_k}(Du^i_{\lambda_0},x) H^i_{\lambda_0,p_j}(Du^i_{\lambda_0},x)v^i_{x_j}(g'(\theta^i_{\lambda_0}))^{-1}]_{x_k}\\&
=-[(g'(\theta^i_{\lambda_0}))^{-1}(v^i+\sum_{j\neq i}(\beta'_{\ep,\lambda_0}(i,j)v^i-\beta'_{\ep,\lambda_0}(i,j)v^j)-f_0^i)]_{x_k}+f^i\\&
+\sum_{j\neq i}[(\beta''_{\ep,\lambda_0}(i,j)\theta_{\lambda_0}^i+\beta''_{\ep,\lambda_0}(j,i)\theta_{\lambda_0}^j)(v^i-v^j)+\beta'_{\ep,\lambda_0}(i,j)f^i-\beta'_{\ep,\lambda_0}(j,i)f^j]+v_0^i.
\end{split}
\eeqs
For any $i=1,\ldots, d$, the right-hand side of the previous equation is in $L^2(\Tt^N)$. 
Thus, the elliptic regularity theory implies that $v\in E^2$. Consequently,  \eqref{fiexplic} gives $f\in E^1.$ 

By induction, if $w_0=(v_0,f_0)\in E^k\times E^{k+1}$, then 
$w_1=(v,f)\in E^{k+2}\times E^{k+1}$. This concludes the proof of the lemma.
\end{proof} 
Claim 3 is now a straightforward consequence of Lemma \ref{Lisomorphlem}  combined with  the implicit function theorem in Banach spaces, see, for instance, \cite{Die}.
\bigskip

Finally, we gather the previous results to prove
the existence claim in
Theorem \ref{main2}.

\begin{proof}[Proof of Theorem \ref{main2} - existence]
Claims 1-3 imply that $\Lambda=[0,1]$. In particular $1\in \Lambda$, i.e., there exists a $C^\infty$ solution of \eqref{obstacleepmfg1sys}-\eqref{obstacleepmfg2sys}.
\end{proof}

\subsection{Uniqueness}

Here, we
complete the proof of Theorem \ref{main2} by 
proving the uniqueness of solutions for \eqref{obstacleepmfg1sys}-\eqref{obstacleepmfg2sys} using the monotonicity method.

\begin{proof}[Proof of Theorem \ref{main2} - uniqueness]
Let $(u_1, \theta_1)$ and $(u_2, \theta_2)$ be  classical  solutions of \eqref{obstacleepmfg1sys}-\eqref{obstacleepmfg2sys}.
First,  we take \eqref{obstacleepmfg1sys}
with $(u,\theta)=(u_1, \theta_1)$ and  $(u,\theta)=(u_2, \theta_2)$ and  subtract the corresponding equations. Next, we multiply by $\theta^i_1-\theta^i_2$ and integrate. Accordingly,  we obtain  
\begin{equation*}\begin{split}&\int_{\Tt^N} [H^i(Du^i_1,x)-H^i(Du^i_2,x)](\theta^i_1-\theta^i_2)+(u^i_1-u^i_2)(\theta^i_1-\theta^i_2)dx
\\&+\sum_{j\ne i}\int_{\Tt^N} (\beta_\ep(u_1^i-u_1^j-\psi^{ij})-\beta_\ep(u_2^i-u_2^j-\psi^{ij}))(\theta^i_1-\theta^i_2)dx\\&
=\int_{\Tt^N}(g(\theta^i_1)-g(\theta^i_2))(\theta^i_1-\theta^i_2)dx.\end{split}\end{equation*}
Now, we take \eqref{obstacleepmfg2sys} with  $(u,\theta)=(u_1, \theta_1)$ and $(u,\theta)=(u_2, \theta_2)$. Next, we subtract the corresponding equations,  multiply by  $ u^i_1-u^i_2,$  and integrate by parts. Accordingly, we get
\begin{equation*}\begin{split}0&=  \int_{\Tt^N} (D_p H^i(D u^i_1, x)\theta^i_1-D_pH^i(Du^i_2,x)\theta^i_2)D(u^i_1-u^i_2)+(u^i_1-u^i_2)(\theta^i_1-\theta^i_2)dx\\&
+\sum_{j\ne i}\int_{\Tt^N}
\left[
\beta'_\ep(u_1^i-u_1^j-\psi^{ij})\theta^i_1-\beta'_\ep(u_1^j-u_1^i-\psi^{ji})\theta^j_1\right.\\&\left.
-\left(\beta_\ep'(u_2^i-u_2^j-\psi^{ij})\theta^i_2-\beta_\ep'(u_2^j-u_2^i-\psi^{ji})\theta^j_2\right)\right](u^i_1-u^i_2)dx.\end{split}\end{equation*}
Finally, we subtract the two previous identities, sum on $i,$ and use  the monotonicity of $g$ to conclude that 
\begin{equation*}\begin{split} 
&0\leq\sum_{i=1}^d\int_{\Tt^N}(g(\theta^i_1)-g(\theta^i_2))(\theta^i_1-\theta^i_2)dx\\&
=\sum_{i=1}^d \int_{\Tt^N} [(H^i(Du^i_1,x)-H^i(Du^i_2,x))(\theta^i_1-\theta^i_2)\\&
-(D_p H^i(D u_1^i, x)\theta_1-D_pH^i(Du_2^i,x)\theta_2)D(u^i_1-u^i_2)]dx\\&
+\sum_{i=1}^d\sum_{j\ne i}\int_{\Tt^N} (\beta_\ep(u_1^i-u_1^j-\psi^{ij})-\beta_\ep(u_2^i-u_2^j-\psi^{ij}))(\theta^i_1-\theta^i_2)dx\\&
-\sum_{i=1}^d\sum_{j\ne i}\int_{\Tt^N}
\left[
\beta'_\ep(u_1^i-u_1^j-\psi^{ij})\theta^i_1-\beta'_\ep(u_1^j-u_1^i-\psi^{ji})\theta^j_1\right.\\&\left.
-\left(\beta_\ep'(u_2^i-u_2^j-\psi^{ij})\theta^i_2-\beta_\ep'(u_2^j-u_2^i-\psi^{ji})\theta^j_2\right)
\right](u^i_1-u^i_2)dx.
\end{split}\end{equation*}
Now, from the convexity of $\beta_\ep$, we infer that 
\begin{equation*}\begin{split}
&\sum_{i=1}^d\sum_{j\ne i}\int_{\Tt^N} (\beta_\ep(u_1^i-u_1^j-\psi^{ij})-\beta_\ep(u_2^i-u_2^j-\psi^{ij}))(\theta^i_1-\theta^i_2)dx\\&
-\sum_{i=1}^d\sum_{j\ne i}\int_{\Tt^N} [\beta'_\ep(u_1^i-u_1^j-\psi^{ij})\theta^i_1-\beta'_\ep(u_1^j-u_1^i-\psi^{ji})\theta^j_1\\&
-(\beta_\ep'(u_2^i-u_2^j-\psi^{ij})\theta^i_2-\beta_\ep'(u_2^j-u_2^i-\psi^{ji})\theta^j_2)](u^i_1-u^i_2)dx\\&
=\sum_{i=1}^d\sum_{j\ne i}\int_{\Tt^N} (\beta_\ep(u_1^i-u_1^j-\psi^{ij})-\beta_\ep(u_2^i-u_2^j-\psi^{ij}))(\theta^i_1-\theta^i_2)dx\\&
-\sum_{i=1}^d\sum_{j\ne i}\int_{\Tt^N} [\beta'_\ep(u_1^i-u_1^j-\psi^{ij})\theta^i_1
-\beta_\ep'(u_2^i-u_2^j-\psi^{ij})\theta^i_2](u^i_1-u^i_2-u_1^j+u_2^j)dx\\&
=-\sum_{i=1}^d\sum_{j\ne i}\int_{\Tt^N}  [\beta_\ep(u_1^i-u_1^j-\psi^{ij})-\beta_\ep(u_2^i-u_2^j-\psi^{ij})\\&-\beta'_\ep(u_2^i-u_2^j-\psi^{ij})(u^i_1-u_1^j-u^i_2+u_2^j)]\theta^i_2 dx\\&
-\sum_{i=1}^d\sum_{j\ne i}\int_{\Tt^N}[\beta_\ep(u_2^i-u_2^j-\psi^{ij})-\beta_\ep(u_1^i-u_1^j-\psi^{ij})\\&-\beta'_\ep(u_1^i-u_1^j-\psi^{ij})(u^i_2-u_2^j-u^i_1+u_1^j)]\theta^i_1 dx\\&
\leq 0.
\end{split}\end{equation*}
Moreover, using \eqref{Hconvexity} of Assumption \ref{A4}, we get 
\begin{equation*}\begin{split} 
&\sum_{i=1}^d \int_{\Tt^N} [(H^i(Du^i_1,x)-H^i(Du^i_2,x))(\theta^i_1-\theta^i_2)\\&
-(D_p H^i(D u_1^i, x)\theta_1-D_pH^i(Du_2^i,x)\theta_2)D(u^i_1-u^i_2)]dx\\&
=-\sum_{i=1}^d\int _{\Tt^N}[H^i(Du^i_1,x)-H^i(Du^i_2,x)-D_p H^i(D u^i_2, x)D(u^i_1-u^i_2)]\theta^i_2dx\\&
-\sum_{i=1}^d\int _{\Tt^N}[H^i(Du^i_2,x)-H^i(Du^i_1,x)-D_p H^i(D u^i_1, x)D(u^i_2-u^i_1)]\theta^i_1dx\\&
\leq -C\sum_{i=1}^d\int _{\Tt^N} |D(u^i_1-u^i_2)|^2dx.
\end{split}\end{equation*}
By combining the last two inequalities, we conclude that 
$$0\leq  -C\sum_{i=1}^d\int _{\Tt^N} |D(u^i_1-u^i_2)|^2dx.$$
Thus, we infer that $u^i_1= u^i_2$ for any $i=1,\ldots,d$.
Consequently, 
\begin{equation*}\begin{split} \theta_1^i&=g^{-1}\left(H^i(Du_1^i,x)+u_1^i+\sum_{j\ne i}\beta_\ep(u_1^i-u_2^j-\psi^{ij})\right)\\&
=g^{-1}\left(H^i(Du_2^i,x)+u_2^i+\sum_{j\ne i}\beta_\ep(u_2^i-u_2^j-\psi^{ij})\right)\\&=\theta_2^i.\end{split}\end{equation*}
This concludes the proof of the uniqueness of the solution of \eqref{obstacleepmfg1sys}-\eqref{obstacleepmfg2sys}.
\end{proof}


\section{Proof of Theorem \ref{main1}}
\label{pmain1}

Now, we use Theorem \ref{main2} and a limiting procedure 
to  prove Theorem  \ref{main1}.

\begin{proof}[Proof of Theorem \ref{main1}]

Let $(u_\ep,\theta_\ep)$ be the classical solution of \eqref{obstacleepmfg1sys}-\eqref{obstacleepmfg2sys}, whose existence is guaranteed by Theorem  \ref{main2}.
By estimates \eqref{w22estlog},   \eqref{calphaestlog}, and  \eqref{w12thetalog}  if Assumption \ref{A5} {\bf L} holds,  or by estimates \eqref{holderestimpower},  \eqref{w22estpower} and \eqref{w12thetapower} if  Assumptions \ref{A5} {\bf P-$\frac 2 N$} and \ref{A6} hold, there exist  $\gamma\in (0,1)$, $u\in (W^{2,2}(\Tt^N))^d\cap (C^\gamma(\Tt^N))^d$, and $\theta\in (W^{1,2}(\Tt^N))^d$, such that, up to extracting a subsequence, as $\ep\to 0$, 
\beqs u_\ep\to u\quad \text{in }(L^\infty(\Tt^N))^d,\eeqs
\beqs Du_\ep\to Du,\quad \theta_\ep\to \theta\quad \text{in }(L^2(\Tt^N))^d,\eeqs
\beqs D^2u_\ep\rightharpoonup D^2u,\quad \text{in }(L^2(\Tt^N))^d,\eeqs and, for any $i=1,\ldots, d$ and $j\neq i,$
\beqs u^i-u^j-\psi^{ij}\leq 0\quad\text{in }\Tt^N.\eeqs
The uniform convergence of $u^i_\ep$ to $u^i$ implies that, if $x\in \Tt^N$ is such that $u^i(x)-u^j(x)-\psi^{ij}(x)<0$, then, for a small enough $\ep$, we have 
$u_\ep^i(x)-u_\ep^j(x)-\psi^{ij}(x)<0$. Consequently, $$\beta_\ep(u_\ep^i(x)-u_\ep^j(x)-\psi^{ij}(x))=\beta'_\ep(u_\ep^i(x)-u_\ep^j(x)-\psi^{ij}(x))=0.$$
We deduce that the limit, $(u,\theta),$ is a weak solution of
\beqs H^i(Du^i,x)+u^i\leq g(\theta^i)\quad\text{in }\Tt^N,\eeqs
\beqs H^i(Du^i,x)+u^i= g(\theta^i)\quad\text{in }\cap_{j\neq i}\{u^i-u^j-\psi^{ij}<0\}.
\eeqs
Next, for $j\neq i$, we introduce the measures
\[
\nu_\epsilon^{ij}=\beta'_\ep(u_\ep^i(x)-u_\ep^j(x)-\psi^{ij}(x))\theta^j.
\]
By \eqref{btl45}, we have that $\int_{\Tt^N}\nu_\epsilon^{ij} \, dx\leq C$ for some constant, $C$, independent of $\epsilon$. Thus, there exist non-negative measures, $\nu^{ij}$, such that
\[
-\di (D_p H^i(Du^i,x)\theta^i)+\theta^i+\sum_{j\neq i}\left(\nu^{ij}-\nu^{ji}\right)=1.
\]  
Moreover, $\nu^{ij}$\ is supported in the set $u^i-u^j-\psi^{ij}=0$. 
\end{proof}

\section{Uniqueness of the limit}
\label{secunlim}

In this last section, we discuss the uniqueness of the limit, $(u, \theta)$, in the proof of Theorem \ref{main1}. 

\begin{pro}
Suppose that Assumptions \ref{A1}-\ref{A4} and \ref{B1}
hold, and that either

\begin{itemize}
\item[-]  Assumption \ref{A5} {\bf L} or
\item[-]  Assumptions \ref{A5} {\bf P-$\frac 2 N$}, \ref{A6} and \ref{A7}

\end{itemize}
hold. For $\epsilon>0$, let $(u_{\epsilon}, \theta_{\epsilon})$
be the solution of  \eqref{obstacleepmfg1sys}-\eqref{obstacleepmfg2sys}.
Then, the limit, $(u, \theta)$, as $\epsilon\to 0$ of the family
$(u_{\epsilon}, \theta_{\epsilon})$ exists; that is, $(u, \theta)$ is independent of the choice of subsequence. 
\end{pro}
\begin{proof}
For $k=0,1$, consider  a sequence, $\epsilon_n^k,$ converging to $0$. Let $(u_k, \theta_k)=\lim_{n\to \infty} (u_{\epsilon_n^k}, \theta_{\epsilon_n^k})$. 
By \eqref{obstacleepmfg1}, we have
\[
H^i(Du^i_k,x)+u^i_k- g(\theta^i_k)\leq 0
\]
and
\[
u^i_k-u^j_k-\psi^{ij}\leq 0. 
\]
For $k=0,1$, let $\tilde k=1-k$. Taking into account the preceding inequalities and the uniform convexity of $H^i$, we have 
\begin{align*}
        0\geq &H^i(Du^i_k,x)- g(\theta^i_k)+u^i_k+
        \sum_j\beta_{\epsilon_n^{\tilde k}}(u^i_k-u^j_k-\psi^{ij})\\\geq& H^i(Du^i_{\epsilon_n^{\tilde k}},x)- g(\theta^i_{\epsilon_n^{\tilde k}})+u^i_{\epsilon_n^{\tilde k}}+\sum_j\beta_{\epsilon_n^{\tilde k}}(u^i_{\epsilon_n^{\tilde k}}-u^j_{\epsilon_n^{\tilde k}}-\psi^{ij})\\ &+u^i_k-u^i_{\epsilon_n^{\tilde k}}
        + D_p
        H^i(Du^i_{\epsilon_n^{\tilde k}},x)\cdot D(u^i_k-u^i_{\epsilon_n^{\tilde k}})\\&+
        \sum_{j\neq i}\beta_{\epsilon_n^{\tilde k}}'(u^i_{\epsilon_n^{\tilde k}}-u^j_{\epsilon_n^{\tilde k}}-\psi^{ij})((u^i_k-u^i_{\epsilon_n^{\tilde k}})-(u^j_k-u^j_{\epsilon_n^{\tilde k}}))\\ &+c\gamma| D(u_k^i-u_{\epsilon_n^{\tilde k}}^i)|^2+g(\theta^i_{\epsilon_n^{\tilde k}})-g(\theta^i_k) 
\end{align*}
for some $c>0$. Integrating with respect to $\theta^i_{\epsilon_n^{\tilde k}}$
and adding over $i$ and $k$ give
\[
\sum_{i,k}\int_{\Tt^N}\left[\gamma| D(u_k^i-u_{\epsilon_n^{\tilde k}}^i)|^2\theta^i_{\epsilon_n^{\tilde
                k}}
+(g(\theta^i_{\epsilon_n^{\tilde
                k}})-g(\theta^i_k))\theta^i_{\epsilon_n^{\tilde
                k}}+u^i_k-u^i_{\epsilon_n^{\tilde k}}\right]dx\leq 0. 
\]
Because $\theta^i_{\epsilon_n^{\tilde
                k}}$ is bounded in $(W^{1,2}(\Tt^N))^d$ and because
$u_{\epsilon_n^{\tilde k}}$ is bounded in $(W^{2,2}(\Tt^N))^d$, by extracting a further subsequence, if necessary, we have
the following almost everywhere convergences:\[
|D(u_k^i-u_{\epsilon_n^{\tilde k}}^i)|^2\theta^i_{\epsilon_n^{\tilde
                k}}
\to| D(u_k^i-u_{\tilde k}^i)|^2\theta^i_{\tilde k} 
\]
and
\[
g(\theta^i_{\epsilon_n^{\tilde
                k}})\theta^i_{\epsilon_n^{\tilde
                k}}\to g(\theta^i_{\tilde
        k})\theta^i_{\tilde
        k}.
\]
Furthermore, 
\[
\int_{\Tt^N}g(\theta^i_k)\theta^i_{\epsilon_n^{\tilde
                k}}dx\to 
\int_{\Tt^N}g(\theta^i_k)\theta^i_{\tilde
        k}dx.
\]
Consequently, taking into account that 
\[
\lim_{n\to \infty}\sum_{k}\int_{\Tt^N}u^i_k-u^i_{\epsilon_n^{\tilde k}}dx=0
\]
and that $z\mapsto z g(z)$ is bounded by below, and using Fatou's Lemma, we obtain
\[
\sum_{i,k}\int_{\Tt^N}\gamma\ |D(u_k^i-u_{\tilde k}^i)|^2\theta^i_{\tilde k}
+(g(\theta^i_{\tilde k})-g(\theta^i_k))\theta^i_{\tilde k}\leq 0.  
\]
Because $g$ is monotone increasing, the preceding inequality implies that $\theta_1=\theta_2$ and that $Du_1=Du_2$. Thus, using \eqref{obstacleepmfg1}, we have $u_1=u_2$. 
\end{proof}

%
%



\bibliographystyle{plain}

\bibliography{mfg}

\end{document}